%% file: prelim_final.tex
\documentclass[11pt,reqno]{amsart}

\usepackage[dvipsnames]{xcolor}
\usepackage{comment}

\newtheorem{theorem}{Theorem}[section]
\newtheorem{proposition}[theorem]{Proposition}
\newtheorem{lemma}[theorem]{Lemma}
\newtheorem{corollary}[theorem]{Corollary}

\newtheorem{assumption}[theorem]{Assumption}
\theoremstyle{definition}
\newtheorem{definition}[theorem]{Definition}
\newtheorem{remark}[theorem]{Remark}
\newtheorem{example}[theorem]{Example}

\numberwithin{equation}{section}

\definecolor{ReferenceGreen}{HTML}{1F6B4F}
\definecolor{CitationRed}{HTML}{A23B3B}
\definecolor{URLBlue}{HTML}{315E85}

\usepackage[
  colorlinks=true,
  linkcolor=ReferenceGreen, 
  citecolor=CitationRed,    
  urlcolor=URLBlue,
  pdfborder={0 0 0}
]{hyperref}

\usepackage[
  nameinlink,
  capitalize,
  noabbrev
]{cleveref}

\usepackage[
  left=27mm,
  right=27mm,
  top=30mm,
  bottom=32mm
]{geometry}

\usepackage{newtxtext,newtxmath} 

\usepackage{parskip}
\input{commands.tex}

\title{Conditioned Brownian motion and local equivalence of path ensembles}
\author{Tobias Schmidt}
\address{Department of Mathematics, TU Darmstadt, Germany.}
\email{tobias.schmidt@tu-darmstadt.de}
\date{\today}

\begin{document}

\begin{abstract}
    We study Brownian motion in $\mathbb R^d$ conditioned so that the time average of a continuous confining potential remains below a fixed level. On every fixed initial time interval, we prove that the conditioned process converges in total variation to the ground-state diffusion associated with a suitable Schrödinger operator. We also obtain sharp asymptotics for the probability of the conditioning event, including bounded perturbations of the constraint. The proof is based on a local limit theorem for the corresponding Feynman-Kac measures. Our results extend the previously known one-dimensional quadratic case to arbitrary finite dimension and a broad class of confining potentials, therefore resolving a conjecture of Aurzada, Lifshits and Schickentanz. 
    The presented approach also works when Brownian motion is replaced by suitable reversible Markov processes, including multidimensional Ornstein-Uhlenbeck processes, CIR processes and continuous-time Markov chains.
\end{abstract}
\maketitle
\section{Introduction and Main Results}

Fix $d\in\mathbb N$, and let $W=(W_t)_{t\geq0}$ be standard Brownian motion in $\mathbb R^d$ with generator $\Delta/2$, started at $x\in\mathbb R^d$. The corresponding probability measure on $C([0,\infty);\mathbb R^d)$ is denoted by $\mathbb P_x$. In case $x=0$ we abbreviate $\mathbb P=\mathbb P_0$. We denote the filtration of the corresponding Brownian motion by $(\mathcal F_t)_{t\ge0}$. Put
\[
 A_T =\int_0^T \de s \, Q(W_s),
\]
where one should intuitively think of $Q\ge 0$ as a confining potential growing at infinity.
In this paper, we are interested in studying Brownian motion conditioned on the event $A_T \le \theta T$ for some $\theta >0$; the corresponding law will be denoted by $\cL_x\bigl((W_t)_{0\leq t\leq S}\mid A_T\leq\theta T\bigr)$. 
Note that the event we condition on is rare.  Indeed, unconditioned Brownian motion explores spatial
scales of order $\sqrt T$, whereas the constraint forces its time-averaged
potential energy to remain of order one.  Since the constraint depends on the
whole path up to time $T$, the conditioned path measure is not Markovian.  Nevertheless, we prove that on every fixed initial time window, a Markovian limit emerges.

 In \cite{ALS}, the authors derive a limit for $\cL_x\bigl((W_t)_{0\leq t\leq S}\mid A_T\leq\theta T\bigr)$ for the specific case $Q(x) = x^2$ in $d=1$.  Using exact small-deviation formulae for
Gaussian quadratic functionals, they showed that the conditioned process
converges weakly to an Ornstein-Uhlenbeck process.  They further conjectured
that, for general confining potentials, the limit is the ground-state
diffusion associated with a suitable Schr\"odinger operator; see
\cite[Conjecture~12]{ALS}.
We prove this conjecture for a broad class of continuous
confining potentials and also strengthen the proposed conclusion in three
directions.  First, on every fixed time window, the convergence holds in total
variation.  Second, we obtain a sharp asymptotic formula for the probability of the conditioning event. Third, we show that the conjecture holds in arbitrary dimensions.

There is also a statistical-mechanical motivation for studying the conditional path law just introduced.  The conditioning
$A_T\leq\theta T$ is a one-sided microcanonical constraint, while
\[
    \frac{\mathrm d\widehat{\mathbb P}^{\beta,T}_x}
         {\mathrm d\mathbb P_x}
    =\frac{1}{Z_T(\beta,x)}\mathrm e^{-\beta A_T},
    \qquad
    Z_T(\beta,x)=\mathbb E^{\prob_x}[\mathrm e^{-\beta A_T}], \footnote{In case $T$ is clear from the context, we will abbreviate $Z(\beta,x) = Z_T (\beta,x)$.} \qquad \beta >0
\]
is the corresponding canonical ensemble.   The two are linked by the exact identity (obtained via the layer-cake representation)
\[
  \mathbb E^{\prob_x}\!\left[\mathrm e^{-\beta A_T}\mathbf 1_F\right]
  =\int_0^\infty \mathrm dy \,\beta \mathrm e^{-\beta y}
      \mathbb P_x(F,A_T\leq y),
      \qquad F\in\mathcal F_T.
\]
Consequently, the canonical law is a mixture of hard-conditioned laws.  We will find that at
large $T$, the mixing measure concentrates around
$y=\lambda'(\beta)T$ on a window of order $\sqrt T$. Here, $\lambda(\beta)$ is the smallest eigenvalue of an associated Schrödinger operator.  This observation alone only yields an averaged form of ensemble equivalence: convergence of the
mixture does not identify the conditional law at the single value
$y=\theta T$.  The local limit theorem proved in this work (see Theorem \ref{thm:llt}) will allow us to undo the mixture and obtain precise information at $y=\theta T$.

The same distinction arises in nonequilibrium statistical mechanics. Broadly speaking, this area studies systems that are not in thermodynamic equilibrium. As an example, consider systems driven by external forces, coupled to reservoirs at different temperatures, transporting particles or energy, or relaxing towards equilibrium \cite{Zwanzig2001,Jack2020}. In stochastic models of such systems, one often asks for the behavior of paths on which a quantity accumulated over time, such as the number of transitions, an integrated current, entropy production, or the time-integrated potential  $A_T \ge 0 $ takes an atypical value \cite{Touchette2018,Jack2020}.
Direct conditioning on an event such as  $\{ A_T\leq\theta T \} $ is generally difficult to analyze. A standard alternative is to replace the original path law by the probability measure obtained by assigning to each path a weight proportional to  $\mathrm e^{-\beta A_T} $, which we call in this paper Feynman-Kac weight. In the physics literature, such measures are called exponentially biased trajectory ensembles, or  $s $-ensembles \cite{LecomteEtAl2007,JackSollich2010,Touchette2018}. 
Equivalence of ensembles is the principle that, for parameters  $\beta $ and  $\theta $ related in the appropriate way, the exponentially biased and the hard-conditioned laws select the same typical long-time behavior \cite{ChetriteTouchettePRL,ChetriteTouchetteAHP,Touchette2011}. This idea has been used to study systems under shear and the slow relaxation of glasses, where transitions may occur between paths with many and few changes of configuration \cite{GarrahanEtAl2007,HedgesEtAl2009,JackSollich2010}. Although the Brownian setting considered here is simpler than the interacting, driven systems typically studied in nonequilibrium statistical mechanics, we obtain precise local equivalence of path ensembles. Our argument extends to a class of reversible Markov processes; see Section \ref{sec:examples}.

The canonical law is useful for another reason: it naturally suggests the Markov process that should describe the conditioned paths. Indeed, the Feynman-Kac semigroup
\[
    f\longmapsto
    \mathbb{E}^{\bbP_x}\!\left[\mathrm e^{-\beta A_t}f(W_t)\right]
\]
loses mass and is therefore not itself a Markov semigroup. Its positive ground state  $\phi_\beta $, however, can be used to transform it into a conservative Markov semigroup. This construction is called the ground-state transform, or the generalized Doob transform in the physics literature \cite{ChetriteTouchettePRL,Garrahan2016}. Under the transformed process, the paths favored by the weight  $\mathrm e^{-\beta A_T} $ become typical in the long-time limit. Consequently, its modified drift or transition rates describe the change in the original dynamics required to sustain the unusual value of  $A_T $ \cite{JackSollich2010,JackSollich2015}. The transformed process also provides a dynamics from which such paths can be simulated without waiting for the rare event to occur under the original law \cite{CauserEtAl2021}.

In systems with many interacting components, the limiting principal eigenvalue or its eigenstate can change abruptly as the bias parameter varies; see e.g. \cite{BeSchSe24}. This signals a qualitative change in the typical biased paths, for example, from a highly active regime to an almost inactive one, and is called a dynamical phase transition \cite{LecomteEtAl2007,GarrahanEtAl2007,Jack2020}. In some models, the tilted operator can also be rewritten as a (quantum) Hamiltonian, so that this phenomenon is related to a ground-state phase transition of the associated quantum system \cite{JackSollich2010}. In our Brownian setting, the transformed process is precisely the diffusion with drift  $\nabla\log\phi_\beta $ and invariant density  $\phi_\beta^2 $, and Theorem~\ref{thm:main} shows that it also describes the fixed-window limit under the hard constraint.

We focus on confining potentials  $Q $, meaning that
 
 $Q(x)\to\infty $ as  $\lvert x\rvert\to\infty $ in $\mathbb R^d$. The constraint
 $A_T\leq \theta T $ then forces the Brownian path to spend an
atypically large proportion of time in a bounded region. At the
logarithmic scale, this localization can be described using the
Donsker-Varadhan theory of large deviations for occupation measures
\cite{DV75a,DV75b}. 
Large deviations theory alone is, however, not sufficiently precise
for the questions considered here. 
It therefore gives neither the  $T^{-1/2} $ prefactor
in Theorem \ref{thm:main} nor the effect of replacing  $\theta T $ by
 $\theta T+z $ with  $z=O(1) $. More importantly, it does not provide
the relative asymptotics of probabilities started from different
points that are needed to identify the conditional law on a fixed
time window. Obtaining this finer information is precisely the role
of the local limit theorem proved below.
We now state our precise assumptions on the potential $Q$.

\begin{assumption}\label{def:classQ}
We say that $Q$ belongs to $\cQ$ if
\begin{enumerate}
  
 \item $Q:\mathbb R^d \to [0,\infty)$ is continuous and is not constant;
 \item for some $p,c,R>0$, \label{bp:ass_2}
       \[
       Q(x)\geq c|x|^p \qquad \text{for $|x| \ge R$}.
       \]
\end{enumerate}
\end{assumption}
Note that $Q$ is locally bounded by continuity and thus locally integrable over $\mathbb R^d$. Define $q_* = \min_{x\in\mathbb R^d}Q(x)$. Moreover, let $\phi_\beta$ be the normalized positive $L^2(\mathbb R^d)$-function that solves
\[
    -\frac12\Delta\phi_\beta+\beta Q\phi_\beta
    =\lambda(\beta)\phi_\beta,
\]
where $\lambda(\beta)$ is the smallest eigenvalue of $H_\beta=-\Delta/2+\beta Q$. All of these objects will be introduced again in detail in Section \ref{sec:main_sec}.

The rate function governing the lower tail of  $A_T/T $ is
\[
  I_Q(\theta)
  =
  \sup_{\beta\geq 0}
  \bigl\{
    \lambda(\beta)-\beta\theta
  \bigr\},
  \qquad \theta>q_*,
\]
where we set  $\lambda(0)=0 $. Equivalently, the
Donsker-Varadhan and Rayleigh-Ritz variational formulae give
\[
  I_Q(\theta)
  =
  \inf
  \left\{
    \frac{1}{2}
    \int_{\mathbb{R}^d} \de x \,
      \lvert\nabla f(x)\rvert^2
    :
      f\in H^1(\mathbb{R}^d), \lVert f\rVert_2=1,
      \int_{\mathbb{R}^d} \de x \, Q(x)|f(x)|^2\leq\theta
  \right\};
\]
see \cite{DV75a,DV75c}.
As shown in Lemma~\ref{lem:range-constraint-coupling}, for every  $\theta>q_* $ the supremum is
attained at the unique  $\beta_\theta>0 $ satisfying
$\lambda'(\beta_\theta)=\theta$.
The corresponding variational optimizer is
 $\phi_{\beta_\theta} $ 
and
\[
  I_Q(\theta)
  =
  \lambda(\beta_\theta)-\beta_\theta\theta
  =
  \frac{1}{2}
    \int_{\mathbb{R}^d} \de x \,
    \lvert\nabla\phi_{\beta_\theta}(x)\rvert^2.
\]
Although this identifies the typical occupation under the
rare event, it does not by itself determine the local conditional
dynamics. For the canonical Feynman-Kac tilt, however, the appropriate
dynamics are suggested through its ground-state transform: it is the reversible
diffusion with invariant probability measure  $\phi_{\beta_\theta}^2(x) \, \de x $, namely
\[
  \de X_t
  =
  \de B_t
  +\nabla\log\phi_{\beta_\theta}(X_t)\,\de t.
\]
Theorem~\ref{thm:main} shows that this canonical candidate is indeed the
fixed-window limit of the microcanonically conditioned Brownian
motion. The fact that the asymptotic is locally uniform is not relevant for the conclusion that the conditioned process asymptotically solves the given SDE. 
\begin{theorem}
\label{thm:main}
Let $Q\in\cQ$, $x\in\mathbb R^d$, and $\theta>q_*$.  Let
$\beta=\beta_\theta$ be the unique solution of
$\lambda'(\beta)=\theta$.  Then $\sigma_\theta^2 = - \lambda''(\beta_\theta)>0$, and for every fixed
$z\in\R$,
\begin{equation}
\label{equ:precise_asymptotics}
 \prob_x(A_T\leq \theta T+z)
 \sim
 \frac{\phi_{\beta_\theta}(x)\Vert\phi_{\beta_\theta}\Vert_1}
      {\beta_\theta\sqrt{2\pi\sigma_\theta^2T}}
 \exp(-I_Q(\theta)T+\beta_\theta z).
\end{equation}
The asymptotic is locally uniform in $(x,z)$.
In particular, for every $S<\infty$,
\begin{equation}
    \nonumber
 \cL_x\bigl((W_t)_{0\leq t\leq S}\mid A_T\leq\theta T\bigr)
 \longrightarrow
 \cL_x\bigl((X_t)_{0\leq t\leq S}\bigr),
\end{equation}
in total variation (TV), where
\[
 \de X_t=\de B_t+\nabla\log\phi_{\beta_\theta}(X_t)\,\de t,
 \qquad X_0=x.
\]
Here the SDE is understood in the weak sense through the
ground-state transform described in
Proposition~\ref{prop:ground-state-transform-rigorous}.
\end{theorem}

\begin{remark}
    The proposed limiting diffusion is  also known as the ground state transform, or $P(\phi)_1$ process associated with $H_{\beta_\theta}$ \cite{Simon79,LHB20}. Moreover, the fact that there exists $\beta_\theta >0$ such that $\lambda' (\beta_\theta) = \theta$ is part of the claim; see also Lemma \ref{lem:range-constraint-coupling}.
\end{remark}
The resulting quantities in Theorem \ref{thm:main} can be computed somewhat explicitly in special cases.
\begin{example}
    Let $Q(x)=|x|^p$ for $p>0$, $x\in\mathbb R^d$, and let $e_{p,d}$ be the smallest eigenvalue of $-\frac12\Delta+|x|^p$.
Scaling the system with the unitary operator 
\[U_a : L^2(\bbR^d)\to L^2(\bbR^d), \qquad (U_af)(x) = a^{d/2}f(ax)\]
gives $\lambda(\beta)=e_{p,d}\beta^{2/(p+2)}$.
For every $\theta>0$, there is a unique
\begin{equation}
 \beta_\theta
 =\left(\frac{2e_{p,d}}{(p+2)\theta}\right)^{(p+2)/p}
  \nonumber
\end{equation}
such that $\lambda'(\beta_\theta)=\theta$.  Moreover,
\begin{equation}
 I_p(\theta)
 =\frac{p}{p+2}e_{p,d}
  \left(\frac{2e_{p,d}}{(p+2)\theta}\right)^{2/p}
  \nonumber
\end{equation}
and
\begin{equation}
 \sigma_{p,\theta}^2
 =\frac{2p}{(p+2)^2}e_{p,d}
  \beta_\theta^{-2(p+1)/(p+2)}.
  \nonumber
\end{equation}
If $\phi_1$ is the normalized ground state at coupling one, then
\begin{align}
    \nonumber
 \phi_\beta(x)
 &=\beta^{d/(2(p+2))}
   \phi_1\bigl(\beta^{1/(p+2)}x\bigr),
\end{align}
and so
\[
 \nabla\log\phi_\beta(x)
 =\beta^{1/(p+2)}
   \nabla\log\phi_1\bigl(\beta^{1/(p+2)}x\bigr).
\]
Inserting $\beta_\theta^{1/(p+2)} = (2e_{p,d}/ [(p+2)\theta])^{1/p}$ yields that the resulting diffusion is
\[\de X_t = \de B_t + a_\theta \nabla \log \phi_1(a_\theta X_t) \de t, \qquad a_\theta = \left( \frac{2e_{p,d}}{(p+2)\theta}\right)^{1/p}.\]
\end{example}
 Theorem \ref{thm:main} solves the conjecture raised in \cite{ALS}, which was stated in $d=1$. Moreover, Theorem \ref{thm:main} also verifies that the microcanonical ensemble, the canonical ensemble and the ground state transform locally and asymptotically describe the same process, which yields an equivalence of ensembles.
 
 The locally uniform tail asymptotics of Theorem~\ref{thm:main} can also be used to treat fixed-width microcanonical windows by
taking differences.  Although the leading probability depends on the location and width of the window, this dependence cancels from the conditional path law, so every fixed window produces the same ground-state diffusion. This is another verification that fixing a specific `energy' for the system corresponds locally and asymptotically to the ground state transform/the Gibbsian reweighting.
\begin{corollary}
\label{cor:fixed-width-window}
Let $Q\in\mathcal Q$, $x\in\mathbb R^d$, and $\theta>q_*$, and use the
notation of Theorem~\ref{thm:main}.  For fixed $a<b$ we find
\[
\mathbb P_x(\theta T+a<A_T\leq \theta T+b)
\sim
\frac{\phi_{\beta_\theta}(x)\|\phi_{\beta_\theta}\|_1}
     {\sqrt{2\pi\sigma_\theta^2T}}\,
\mathrm e^{-I_Q(\theta)T}
\int_a^b \de y \, \mathrm e^{\,\beta_\theta y} .
\]
The asymptotic is locally uniform in
\[
(x,a,b)\in
\bbR^d \times\bigl\{(a,b)\in\mathbb R^2:a<b\bigr\}.
\]
Moreover, for every $S<\infty$,
\[
\mathcal L_x\bigl((W_t)_{0\leq t\leq S}\mid \theta T+a<A_T\leq \theta T+b\bigr)
\xrightarrow[T\to\infty]{\mathrm{TV}}
\mathcal L_x\bigl((X_t)_{0\leq t\leq S}\bigr),
\]
where $X$ is the ground-state diffusion as in Theorem~\ref{thm:main}.  In particular,
the limiting process is independent of the choice of $a<b$.
\end{corollary}

As noted earlier, we derive Theorem \ref{thm:main} by verifying a local limit theorem (LLT) for the family of path measures
\begin{equation}
    \label{equ:FK_path_measures}
    \widehat\bbP_x ^{\beta,T} (\de W)= \frac{1}{Z(\beta,x)} \exp \lt( -\beta \int_0 ^T \de s \,Q(W_s)\rt)\bbP_x (\de W).
\end{equation}

\begin{theorem}[Local limit theorem]
\label{thm:llt}
Fix $\beta >0$.
Let $Y_T=A_T-\theta_\beta T$ with $\theta_\beta = \lambda'(\beta)$ and $\sigma^2_\beta = - \lambda''(\beta)$.  For every directly Riemann-integrable
$h:\mathbb R\to\mathbb R$, every compact $K\subset\mathbb R^d$, and every $M<\infty$,
\begin{equation}
    \nonumber
 \sup_{\substack{x\in K\\|a|\leq M}}
 \left|
  \sqrt T\,\bbE^{\bbPh^{\beta,T}_x}[h(Y_T-a)]
  -\frac1{\sqrt{2\pi\sigma_\beta^2}}
   \int_\R \de y \, h(y)
 \right|
 \longrightarrow0.
\end{equation}
In particular,
\begin{equation}
\nonumber
 \sqrt T\bbE^{\bbPh^{\beta,T}_x}[h(Y_T)]
 \longrightarrow
 \frac1{\sqrt{2\pi\sigma_\beta^2}}
 \int_\bbR \de y \, h(y),
\end{equation}
locally uniformly in $x$.
\end{theorem}
\subsection{Proof overview}
\label{sec:proof_idea}
We now give a high-level overview of our proposed approach and how a LLT for the family $\bbPh^{\beta,T}_x$ yields Theorem \ref{thm:main}. Compared to the proof in \cite{ALS}, we are not in the situation where small ball probabilities are given in closed-form, and explicit calculations are not possible. Instead, we rely on spectral theory for Schrödinger operators. As a first step, we replace the microcanonical conditioning on the set $\{ A_T \le \theta T \}$ by a soft constraint. Indeed, we study the probability measures $\bbPh^{\beta,T}_x$ introduced in equation \eqref{equ:FK_path_measures}. This is very convenient, because this measure is the path measure corresponding to the Schrödinger operator $H_\beta = - \Delta/2 + \beta Q$. By our standing assumptions introduced in Assumption \ref{def:classQ}, $H_\beta$ has discrete spectrum due to the confining nature of $Q$. This allows us to study the Fourier transform of the random variable $A_T$ under $\bbPh^{\beta,T}$ by means of perturbation theory. Indeed, note the identity 
\begin{equation}
    \nonumber
    \Theta_{\beta,T}(u) = \bbE^{\bbPh^{\beta,T}}[\exp(iuA_T)] = \frac{Z(\beta-iu,0)}{Z(\beta,0)}.
\end{equation}
Let $\lambda(\beta) = \inf \spec H_\beta$ and assume that  $H_{\beta-iu}= -\Delta/2 + (\beta-iu)Q$ corresponds in an appropriate sense to the respective (complex) Feynman-Kac weight.
If one believes in the heuristic that $Z(\beta-iu,x)=[\exp(-TH_{\beta - iu}) \mathbf 1](x)$ behaves asymptotically as $\exp(-T\lambda(\beta - iu))$ and that $z \to \lambda(z)$ is analytic around a small neighborhood of $\beta$, then a simple Taylor expansion yields
\[\Theta_{\beta,T}(u) = \exp \Big(  T \big( iu \lambda'(\beta) + \frac{u^2}{2} \lambda''(\beta)+O(|u|^3) \big) \Big).\]
After shifting $A_T$ to be centered at leading order, meaning that we  study the random variable  
\[Y_T = A_T - \lambda'(\beta)T\]
and putting $u = v/ \sqrt{T}$, we see that a CLT is to be expected for $Y_T$. This is, however, insufficient for our purposes. First, the CLT is derived for the tilted measure, which is not our original interest. Second, having a CLT does not provide enough information to determine ratios of half-ray probabilities
$\bbP_x( A_T \le \theta T + z)$, which is what determines the limiting process. We solve this problem by verifying a LLT for $Y_T$ under $\bbPh^{\beta_\theta,T}$ (see Theorem \ref{thm:llt}), where $\beta_\theta>0$ is chosen such that $A_T - \theta T$ is centered at leading order (i.e. $\lambda '(\beta_\theta)= \theta$). By applying the LLT (Theorem \ref{thm:llt}) with the function
\[h(y) = \mathbf 1_{y \le z }\exp(\beta y),\] 
 we can `untilt' $\bbPh^{\beta,T}$ in order to verify \eqref{equ:precise_asymptotics}. 
A local limit theorem is natural here for the same reason that it is natural
for additive functionals of 
a continuous, stationary
process.  Under the canonical
tilt, the bulk of a long path is governed by the positive recurrent
ground-state diffusion, and $A_T$ is an additive functional of that process.
One therefore expects Gaussian fluctuations, together with a local limit
theorem provided that no nonzero Fourier mode remains coherent.
The spectral proof provided in this paper makes this intuition precise.  Near $u=0$, another expansion of $\lambda$ at $\beta$ gives
\[
  \lambda(\beta-iu)
  =\lambda(\beta)-iu\lambda'(\beta)
     +\frac12\sigma_\beta^2u^2+O(|u|^3),
  \qquad \sigma_\beta^2=-\lambda''(\beta)>0.
\]
After centering, the linear term cancels, and the ground-state mode therefore
converges to the Gaussian Fourier transform on the scale
$u=O(T^{-1/2})$.  The ordinary spectral gap of $H_\beta$ suppresses the
remaining spectral subspace near zero frequency.  For a local limit theorem,
however, this ordinary gap is not sufficient: one must also exclude coherent
nonzero Fourier modes.  One main ingredient, Lemma \ref{lemma:strict_spectral_gap}, provides exactly this input by showing that, for every $u\neq0$,
\[
   \inf\{\Re\zeta:
             \zeta\in\operatorname{spec}(H_{\beta-iu})\}
       >\lambda(\beta),
\]
with a uniform gap when $u$ ranges over a compact annulus
$\delta\leq|u|\leq R$.
There is a simple physical interpretation of this strict inequality.
The imaginary potential $-iuQ$ rotates the phase at a rate depending on the
position.  Equality of the real spectral energy with $\lambda(\beta)$ would,
by the diamagnetic and Rayleigh-Ritz inequalities, force the modulus of the
corresponding eigenfunction to equal $\phi_\beta$ and its phase to be
constant, as position-dependent phase shifts cost kinetic energy.  The eigenvalue equation would then force $Q$ to be constant,
contrary to our assumptions. Fourier modes close to zero can be controlled by analytic perturbation theory (cf. Proposition \ref{prop:analytic-band}).  Thus, every nonzero Fourier mode decays exponentially.  Fourier
inversion then leaves only the
Gaussian zero-frequency mode and yields Theorem~\ref{thm:llt}.

\subsection{Related work}

Our starting point is the work of Aurzada, Lifshits, and Schickentanz
\cite{ALS}.  In one dimension, for $Q(x)=x^2$, they combine exact
Gaussian small-deviation asymptotics with a disintegration argument to obtain
the Ornstein-Uhlenbeck limit.  Their proof exploits formulae special to
quadratic functionals.  The present work proves their general-potential
conjecture, along with a sharp half-line asymptotic.

Sharp large-deviation asymptotics beyond the logarithmic scale have been
studied in several general frameworks.  In
\cite{ChagantySethuraman}, the authors established `strong' large-deviation and local limit
theorems for general sequences of random variables, while  \cite{KontoyiannisMeyn} developed spectral exact large-deviation
asymptotics for additive functionals of geometrically ergodic Markov
processes; see also \cite{KusuokaTamura} for precise
estimates of Donsker-Varadhan type in case the Markov process under consideration has an invariant probability measure.  
Most directly related to our setting are works of Fatalov. In \cite{Fatalov}, the author obtains, for multidimensional
Brownian motion under a polynomial-coercivity assumption essentially
comparable to Assumption~\ref{def:classQ}, an exact lower-tail asymptotic
corresponding to the  $z=0 $ case of our sharp half-line asymptotic.
Related exact asymptotics for  $L^p $-functionals of
Ornstein-Uhlenbeck processes were obtained in
\cite{FatalovOU}; these overlap with special cases of the sharp-tail
calculations in Section~\ref{sec:examples}.
In the proof of \cite{Fatalov}, to the best of our understanding, the passage to the
 $T^{-1/2} $ prefactor uses a central limit theorem inside a
 $T $-dependent Laplace integral whose mass is concentrated in an
 $O(T^{-1/2}) $ neighborhood of zero. This requires additional control, which does not follow from the stated central
limit theorem alone (see \cite[Lemma $4.8$]{Fatalov}, and in particular equation (4.50)). The local limit theorem proved below supplies
this control under Assumption~\ref{def:classQ} and, in addition, yields the locally
uniform  $O(1) $-shift asymptotics needed for the fixed-window
total-variation limit.

Exponential weighting of Wiener measure is classical. Limits of Brownian motion penalized by normalized exponential weights were derived in considerable generality in \cite{RoynetteValloisYor}; see also
\cite{LHB20}.  On the operator side, the same limiting object
is the $P(\phi)_1$ process obtained by a ground-state transform of a
Schr\"odinger semigroup; see  \cite{Simon79,RosenSimon}.  These works always study the canonical ensemble, but do not by themselves justify replacing the exponential weight by the hard event $\{ A_T\leq\theta T\}$: the Laplace-mixture identity above loses the information required to isolate one energy density.

The microcanonical-canonical terminology belongs to the broader theory of
ensembles of trajectories.  Chetrite and Touchette
\cite{ChetriteTouchettePRL,ChetriteTouchetteAHP} formulate microcanonical and
canonical path ensembles for Markov processes and construct the associated
driven process by a generalized Doob transform.  Under convexity assumptions
they prove equivalence at the logarithmic large-deviation level, which implies
equality of typical macroscopic behavior.  Their framework correctly predicts
the ground-state diffusion appearing here.  Its notion of equivalence is,
however, weaker than the fixed-window total-variation convergence and sharp
$T^{-1/2}$ asymptotics proved in the present paper; in particular,
logarithmic equivalence alone does not give the fixed $O(1)$ shift ratios used
in our Markov property argument.

The Fourier proof presented in this work belongs to the spectral approach to limit theorems for
additive functionals of Markov processes.  The method goes back to Nagaev and
was developed through perturbation theory for Fourier kernels; 
references include \cite{HennionHerve,HervePene,FerreHerveLedoux}.
These works isolate the two ingredients beyond a
central limit theorem: perturbative control near zero frequency and sufficiently fast decay away from zero.  Our proof realizes the
Fourier kernels directly as the complex Schr\"odinger semigroups
$\mathrm e^{-tH_{\beta-iu}}$ and treats the endpoint weights present in the
finite-horizon Feynman-Kac law.  The final step then follows classical ideas similarly to \cite{StoneLLT,StoneRatio}. See also \cite{vershynin2026friendly} for a very nice introduction to some of those ideas.
We also want to mention the work \cite{RoynetteYorLLT}, whose title concerns
local limit theorems for Brownian additive functionals.  Their regime is
different from the one considered here: they study unpenalized additive
functionals generated by integrable or one-sided potentials and obtain a
$\sqrt T$-scaled vague limit on the fixed-value scale.  In contrast, our local
limit theorem concerns Gaussian fluctuations around the extensive value
$\lambda'(\beta)T$ under a confining Feynman-Kac tilt and is used to derive a
sharp rare-event asymptotic for the original Brownian law.

Finally, the quadratic case is naturally connected with the theory of small
deviations for Gaussian processes; see Lifshits \cite{Lif} and
the survey of Nazarov and Petrova \cite{NazarovPetrova}.  The spectral method
used here retains the sharp information needed for conditioning without
requiring $A_T$ to be a quadratic Gaussian functional.

\subsection{Organization and nomenclature}
We define $\Re z$ to be the real part of $z\in\mathbb C$; the imaginary part is denoted by $\Im z$. By $\lambda(\beta)$ we denote the smallest eigenvalue of $H_\beta$, $\beta>0$. The norm $\|\cdot\|_p$ refers to the norm on $L^p(\mathbb R^d)$; in case no subscript is present, the inner product $\langle\cdot,\cdot\rangle$ refers to the inner product on $L^2(\mathbb R^d)$. The norm $\|\cdot\|_{2\to2}$ is the operator norm on $L^2(\mathbb R^d)$. The symbols $u$ and $v$ used as Fourier variables remain real-valued.

The remaining part of the paper is organized as follows: in Section \ref{sec:main_sec} we provide the main argument by first proving the universal spectral gap. This, together with some other technicalities, then allows us to apply Lemma \ref{lem:stone-smoothing} in order to obtain the LLT. Section \ref{sec:examples} is devoted to generalizing our method to other stochastic processes, including multidimensional Ornstein-Uhlenbeck processes.
Because quite a lot is already known on spectral theory for Schrödinger operators, some results are already known to readers familiar with the topic. We decided to postpone most proofs of simple, or already well-known results in that direction into the Appendix. Our reasoning for including them is twofold: first, most probabilists might not be very familiar with the theory and terminology, and therefore we deemed it best to make the paper as self-contained as possible. Second, we were unable to locate some results as we needed them, for example the complex Feynman-Kac formula obtained in Theorem \ref{thm:complex-fk-app}. It therefore seems like a worthwhile endeavor to collect such statements in our Appendix.

\section{Proof of the Main Results}
\label{sec:main_sec}
Let $\beta >0$. In all that follows we define
\[
 H_\beta=-\frac12\Delta+\beta Q
 \quad\text{on }L^2(\mathbb R^d),
\]
where we use the common form domain 
\[
  \mathcal D
  =H^1(\mathbb R^d)\cap L^2(\mathbb R^d,Q(x)\,\de x).
\]
We will also be interested in studying $H_z$ with $z \in \bbC$ and $\Re z > 0$. The resulting operator is of course not self-adjoint, but can be defined via sectorial forms; for the details we refer to Appendix \ref{sec:sectorial-background}. We will usually write any complex number used to define a Schrödinger operator as $z = \beta - i u$ with  $\beta > 0$ and $u \in \R$. In that case, $\exp(-tH_z)$ is not defined via the functional calculus as it is for $u = 0$, but instead 
as
the semigroup associated 
to the sectorial form 
\begin{equation}
    \nonumber
  \mathfrak h_z[f,g]
  =\frac12\int_{\mathbb R^d} \de x \, \nabla f(x)\cdot\overline{\nabla g(x)}
     +z\int_{\mathbb R^d} \de x \, Q(x)f(x)\overline{g(x)}.
\end{equation}
If $\lambda(\beta)=\inf\spec H_\beta$ and
\begin{equation}
    \nonumber
    H_\beta\phi_\beta=\lambda(\beta)\phi_\beta,
 \qquad \phi_\beta>0,\qquad \norm{\phi_\beta}_2=1,
\end{equation}
then $\phi_\beta$ is called the ground state of $H_\beta$. The fact that $\phi_\beta$ is unique and can be chosen strictly positive is a result from Perron-Frobenius theory; we  refer to Proposition \ref{prop:ground-state-spectral-theory} for the details.
The following lemma allows us to control the characteristic function of $Y_T$ bounded strictly away from zero via \eqref{eq:uniform-gap}.
\begin{lemma}
\label{lemma:strict_spectral_gap}
If $u\neq0$, then every $\zeta\in\spec(H_{\beta-iu})$ satisfies
\begin{equation}
 \Re\zeta>\lambda(\beta).
 \label{eq:strict-gap}
\end{equation}
In particular, for every $0<\delta < R <\infty$,
\begin{equation}
    \label{equ:fourier_annulus_uniform}
 \inf_{\delta\leq|u|\leq R}
 \inf_{\zeta\in\spec(H_{\beta-iu})}\Re\zeta>\lambda (\beta).
\end{equation}
Consequently, for every $0<\delta<R<\infty$, there exist $C,\eta>0$ such
that
\begin{equation}
 \sup_{\delta\leq|u|\leq R}
 \norm{\mathrm e^{-tH_{\beta-iu}}}_{2 \to 2}
 \leq C \mathrm e^{-(\lambda(\beta)+\eta)t},
 \qquad t\geq1.
 \label{eq:uniform-gap}
\end{equation}
\end{lemma}

\begin{proof}
Because $H_{\beta-iu}$ has compact resolvent, every point of its spectrum is
an eigenvalue.  Let
$H_{\beta-iu}\psi=\zeta\psi$, with $\norm\psi_2=1$.  Taking real parts in
the form identity gives
\begin{equation}
 \Re\zeta
 =\frac12\norm{\nabla\psi}_2^2
  +\beta\int_{\mathbb R^d} \de x \, Q(x) |\psi(x)|^2.
 \label{eq:real-eigenvalue-form}
\end{equation}
The diamagnetic inequality
\[
 |\nabla|\psi||\leq|\nabla\psi|\qquad\text{a.e.}
\]
holds for every complex-valued $\psi\in H^1(\mathbb R^d)$; see, for example,
\cite[Theorem~7.21]{LiebLoss} or \cite{SimonSemigroups}.  Rayleigh-Ritz
applied to $|\psi|$ therefore yields
\[
 \Re\zeta
 \geq \frac12\norm{\nabla|\psi|}_2^2
      +\beta\int_{\mathbb R^d} \de x \, Q(x)|\psi(x)|^2
 \geq\lambda(\beta).
\]
Suppose equality held.  Then equality holds in Rayleigh-Ritz, so simplicity
and positivity of the ground state imply
$|\psi|=\phi_\beta$.  Put $q=\psi/\phi_\beta$.  Then
$q\in H^1_{\mathrm{loc}}(\mathbb R^d)$, $|q|=1$, and
\[
 |\nabla\psi|^2
 =|q\nabla\phi_\beta+\phi_\beta\nabla q|^2
 =|\nabla\phi_\beta|^2+\phi_\beta^2|\nabla q|^2
 \qquad\text{a.e.},
\]
because $\Re(\overline q \partial_i q)=\partial_i \frac12|q|^2=0$ for $1 \le i \le d$. Equality in the
diamagnetic step therefore gives
$\int_{\mathbb R^d}\de x\,\phi_\beta(x)^2|\nabla q(x)|^2=0$.  Since $\phi_\beta>0$ and $\mathbb R^d$ is connected, $q$ is constant,
and hence $\psi=e^{i\alpha_0}\phi_\beta$ for some $\alpha_0\in\R$, meaning that $\phi_\beta$ has the same $H_{\beta-iu}$-energy as $\psi$.
Substitution into the eigenvalue equation
gives
\[
 \bigl(\lambda(\beta)-iuQ(x)\bigr)\phi_\beta(x)
 =\zeta\phi_\beta(x)
\]
for almost every $x$.  Since $u\neq0$ and $\phi_\beta>0$, this would make
$Q$ constant almost everywhere, and hence everywhere by continuity.  This
contradicts $Q\in\cQ$ and hence proves \eqref{eq:strict-gap}.

It remains to make the gap uniform on a compact annulus bounded away from $0$.  Suppose
that this were impossible.  Then there would exist
$u_n\to u$ with $\delta\leq|u|\leq R$ and normalized eigenpairs
\[
 H_{\beta-iu_n}\psi_n=\zeta_n\psi_n,
 \qquad \Re\zeta_n\downarrow\lambda(\beta).
\]
Equation \eqref{eq:real-eigenvalue-form} bounds $(\psi_n)$ in $\cD$.
Moreover, taking imaginary parts gives
\begin{equation}
    \nonumber
 \Im\zeta_n=-u_n\int_{\mathbb R^d} \de x \, Q(x)|\psi_n(x)|^2,
\end{equation}
so $(\zeta_n)$ is bounded.  Passing to a subsequence, compactness of
$\cD\hookrightarrow L^2$ gives
$\psi_n\to\psi$ strongly in $L^2$ and weakly in $\cD$, while
$\zeta_n\to\zeta$.  Passing to the limit in
\[
 \mathfrak h_{\beta-iu_n}[\psi_n,f]
 =\zeta_n\ip{\psi_n}{f},
 \qquad f\in\cD,
\]
shows that $H_{\beta-iu}\psi=\zeta\psi$.  Since $\norm\psi_2=1$ and
$\Re\zeta=\lambda(\beta)$, this contradicts \eqref{eq:strict-gap}.  We have thus
proved \eqref{equ:fourier_annulus_uniform}.
Finally, \eqref{eq:uniform-gap} follows from Lemma~\ref{lem:spectral-to-semigroup}, using
\[
 \Re\mathfrak h_{\beta-iu}[f]
 =\mathfrak h_\beta[f]
 \geq\lambda\norm f_2^2.\qedhere
\]
\end{proof}
Next, we want to connect $Z(z,x)$ to the semigroup generated by $H_z$. This is slightly non-trivial, because $\mathbf 1\notin L^2(\mathbb R^d)$. There exist, however, well-known solutions for it.

\begin{lemma}
\label{lem:real-boundary-smoothing}
For every  $\tau>0 $, define
\[
 g_\tau(x)=\mathbb E^{\prob_x}[\mathrm e^{-\beta A_\tau}]
 \quad\text{and}\quad
 \ell_{\tau,x}(f)=(\mathrm e^{-\tau H_\beta}f)(x)
\]
for any $f\in L^2(\mathbb R^d)$.
Then  $g_\tau\in L^2(\mathbb R^d) $, and for $x\in\mathbb R^d$ we have
\begin{equation}
 |\ell_{\tau,x}(f)|\le (4\pi\tau)^{-d/4}\|f\|_2.
 \label{eq:point-evaluation-bound}
\end{equation}
Moreover, for  $T\ge 2\tau $,
\begin{equation}
 Z_T(\beta,x)=\mathbb E^{\prob_x}[\mathrm e^{-\beta A_T}]
 =\ell_{\tau,x}\!\left(\mathrm e^{-(T-2\tau)H}g_\tau\right).
 \label{eq:two-sided-real-smoothing}
\end{equation}
\end{lemma}
\begin{proof}
Choose  $R,c,p>0 $ so that  $Q(y)\ge c|y|^p $ for  $|y|\ge R $.
For  $|x|\ge 2R $, let
\[
 E_x=\left\{\sup_{0\le s\le\tau}|W_s-x|\le |x|/2\right\}.
\]
On  $E_x $, one has  $|W_s|\ge |x|/2 $ for every  $s\le\tau $, and hence
 $A_\tau\ge c2^{-p}\tau|x|^p $.  The reflection principle and a Gaussian
tail bound give
\[
 \mathbb P_x(E_x^c)
 \le 4d\exp \big(-|x|^2/(8d\tau)\big).
\]
Consequently,
\begin{equation}
\nonumber
 0\le g_\tau(x)
 \le \exp(-\beta c2^{-p}\tau|x|^p)
      +4d\exp(-|x|^2/(8d\tau))
 \label{eq:g-tau-decay}
\end{equation}
for all sufficiently large  $|x| $.  Since  $0\le g_\tau\le1 $ on bounded
sets, this proves  $g_\tau\in L^2(\mathbb R^d) $.
Feynman-Kac domination by the free heat semigroup  $P_\tau $, followed by
Cauchy-Schwarz, yields
\[
 |\ell_{\tau,x}(f)|
 \le (P_\tau|f|)(x)
 \le \|p_\tau(x,\cdot)\|_2\|f\|_2
 =(4\pi\tau)^{-d/4}\|f\|_2,
\]
which is \eqref{eq:point-evaluation-bound}.  Finally,
\eqref{eq:two-sided-real-smoothing} follows from the semigroup property after splitting the time interval into pieces of
lengths  $\tau,T-2\tau,\tau $.
\end{proof}

\begin{proposition}
\label{prop:partition-asymptotic-rigorous}
Let  $\lambda_1(\beta)>\lambda(\beta) $ denote the next eigenvalue of
 $H_\beta $.  Then, for every compact  $K\subset\mathbb R^d $,
\begin{equation}
 Z_T(\beta,x)
 =\mathrm e^{-\lambda(\beta)T}\phi_\beta(x)\|\phi_\beta\|_1
 \left(1+O_K\!\left(\mathrm e^{-[\lambda_1(\beta)-\lambda(\beta)]T}\right)\right)
 \label{eq:partition-asymptotic-rigorous}
\end{equation}
uniformly for  $x\in K $.  In particular, the same assertion holds with
an  $O_x(\cdot) $ error for every fixed  $x $.
\end{proposition}

\begin{proof}
Fix  $\tau>0 $ and let  $P f=\langle f,\phi_\beta\rangle\phi_\beta $ be the
orthogonal projection onto the ground state.  The spectral theorem gives,
for  $s\ge0 $,
\[
 \mathrm e^{-sH_\beta}=\mathrm e^{-\lambda(\beta) s}P+R_s,
 \qquad \|R_s\|_{2\to2}\le e^{-\lambda_1 (\beta)s}.
\]
Using Lemma~\ref{lem:real-boundary-smoothing} with  $s=T-2\tau $, we obtain
\begin{align}
 Z_T(\beta,x)
 &=\mathrm e^{-\lambda(\beta)(T-2\tau)}
   \ell_{\tau,x}(\phi_\beta)\langle g_\tau,\phi_\beta\rangle
   +\ell_{\tau,x}(R_{T-2\tau}g_\tau).
 \label{eq:partition-spectral-split}
\end{align}
The eigenfunction identity, in its continuous pointwise version, gives
\[
 \ell_{\tau,x}(\phi_\beta )=\mathrm e^{-\lambda(\beta)\tau}\phi_\beta (x).
\]
Also,
\begin{equation}
 \langle g_\tau,\phi_\beta \rangle
 =\mathrm e^{-\lambda (\beta)\tau}\langle \mathbf 1,\phi_\beta \rangle 
 =\mathrm e^{-\lambda (\beta)\tau}\|\phi_\beta \|_1 .
 \label{eq:ground-boundary-overlap}
\end{equation}
To justify \eqref{eq:ground-boundary-overlap} despite  $\mathbf 1\notin L^2(\mathbb R^d) $, use
the nonnegative symmetric Feynman-Kac kernel  $k_\tau(y,z) $.  Tonelli's Theorem and  $\phi_\beta\in L^1(\mathbb R^d) $ give
\[
 \int \de y \, \phi_\beta (y)\!\int \de z \, k_\tau(y,z)
 =\int \de z \,(\mathrm e^{-\tau H_\beta}\phi_\beta )(z)
 =\mathrm e^{-\lambda\tau}\int \de z\,\phi_\beta (z).
\]
Thus the first term in \eqref{eq:partition-spectral-split} is exactly
 $e^{-\lambda(\beta) T}\phi_\beta (x)\|\phi_\beta \|_1 $.  By
\eqref{eq:point-evaluation-bound}, the remainder satisfies
\[
 \left|\ell_{\tau,x}(R_{T-2\tau}g_\tau)\right|
 \le (4\pi\tau)^{-d/4}\|g_\tau\|_2
      e^{-\lambda_1 (\beta)(T-2\tau)},
\]
uniformly in  $x $.  Since  $\phi_\beta  $ is continuous and strictly positive,
 $\inf_{x\in K}\phi_\beta (x)>0 $.  Dividing by the leading term therefore
proves \eqref{eq:partition-asymptotic-rigorous}, locally uniformly in  $x $.
\end{proof}
The next proposition is a generalization of the previous one, relying on analytic perturbation theory. The key difference is that the constant that multiplies the potential $Q$ can also be complex.
For later use, we record the corresponding complex notation. For  $u\in\mathbb R $, set\[g_{\tau,u}(x)=\mathbb E^{\bbP_x}\left[\exp\left(-(\beta-iu)A_\tau\right)\right],\qquad\ell_{\tau,u,x}(f)=\left(\exp\left(-\tau H_{\beta-iu}\right)f\right)(x).\]
The complex Feynman-Kac formula gives
\[|g_{\tau,u}(x)|\leq g_{\tau,0}(x),\qquad|\ell_{\tau,u,x}(f)|\leq(4\pi\tau)^{-d/4}\|f\|_2.\]
In particular,  $g_{\tau,u}\in L^2(\mathbb R^d) $, uniformly for  $u $ in compact sets. Moreover,
\[\left|\ell_{\tau,u,x}(f)-\ell_{\tau,v,x}(f)\right|\leq\frac{|u-v|}{\mathrm e\beta}\left(P_\tau|f|\right)(x),\]
while  $u\mapsto g_{\tau,u} $ is continuous in  $L^2(\mathbb R^d) $ by dominated convergence. Thus all endpoint quantities used below depend continuously on  $u $, locally uniformly in  $x $. Once  $\tau $ is fixed, we abbreviate
\[g_u=g_{\tau,u},\qquad\ell_{u,x}=\ell_{\tau,u,x}.\]
\begin{proposition}
\label{prop:analytic-band}
There exist $\varepsilon,\gamma,C>0$, an analytic eigenvalue
$u\mapsto\lambda_u =\lambda(\beta-iu)$, and analytic rank-one Riesz
projections $u\mapsto P_u$, defined on an open neighborhood of $|u|\le\varepsilon$, such that
\begin{equation}
 S_u(t)=e^{-t\lambda_u}P_u+R_u(t),
 \qquad
 \norm{R_u(t)}_{2\to2}
 \leq Ce^{-(\lambda (\beta)+\gamma)t}
 \label{eq:local-spectral-decomposition}
\end{equation}
for $|u|\leq\varepsilon$ and $t\geq1$.  Moreover,
\begin{equation}
 \lambda(\beta-iu)
 =\lambda (\beta)-iu \lambda'(\beta)
  -\frac12 \lambda''(\beta) u^2+O(|u|^3).
 \label{eq:lambda-expansion}
\end{equation}
After decreasing $\varepsilon$ if necessary,
\begin{equation}
 \Re\bigl(\lambda(\beta-iu)-\lambda(\beta)+iu \lambda'(\beta)\bigr)
 \geq c_0u^2,
 \qquad |u|\leq\varepsilon,
 \label{eq:quadratic-real-part}
\end{equation}
for some $c_0>0$.
\end{proposition}

\begin{proof}
The forms $z\mapsto\mathfrak h_z$ constitute a holomorphic family of type
(a) on the common domain $\cD$, and the associated operators form a
holomorphic family of type (B).  Since $\lambda (\beta)$ is an isolated simple
eigenvalue of $H_\beta$, analytic perturbation theory for isolated
eigenvalues gives analytic continuations $\lambda(z)$ and $P(z)$ in a
neighborhood of $\beta$; see
\cite[Chapter~VII, Sections~1 and~4]{Kato}.  Taylor expansion along
$z=\beta-iu$ gives
\eqref{eq:lambda-expansion} and then
\eqref{eq:quadratic-real-part}.

It remains to verify \eqref{eq:local-spectral-decomposition}. The difficulty here is that the real gap between the first and second eigenvalues of small perturbations might not be uniformly big. At $u=0$, the self-adjoint operator $H_{\beta}$ possesses a strict spectral gap; the spectrum on $\text{Ran}(I-P_0)$ lies entirely in $[\lambda+\gamma_{0},\infty)$ for some $\gamma_{0}>0$. To bound $R_u(t)$ via Lemma \ref{lem:spectral-to-semigroup}, we must prove that a uniform positive real-part gap between the first and second eigenvalue survives for all sufficiently small perturbations $|u| \le \varepsilon$.
Norm-resolvent continuity of a type-(B) family guarantees spectral stability across curves in the complex plane which are compactly supported; see also the discussion in Appendix \ref{subsec:holomorphic-riesz-app}.
To  exclude eigenvalues coming down from infinity, note that for any eigenvalue $\zeta \in \text{spec}(H_{\beta-iu})$ with normalized eigenfunction $\psi$, evaluating the imaginary part of the associated form yields
\begin{equation*}
|\mathfrak{I}\zeta|=|u|\int_{\mathbb{R}^d}\de x \, Q(x)|\psi (x)|^{2}\le\frac{|u|}{\beta}\mathfrak{R}\zeta.
\end{equation*}
This inequality confines the entire spectrum of $H_{\beta-iu}$ to a forward-facing cone in the complex plane. Consequently, if we restrict our attention to a fixed vertical strip where the real part $\mathfrak{R}\zeta$ is bounded, the imaginary part $\mathfrak{I}\zeta$ is bounded, as well.
Thus, all spectral points that could potentially threaten the uniform gap remain confined within a strictly bounded, compact subset of $\mathbb{C}$. Therefore, \eqref{eq:local-spectral-decomposition} follows from Lemma \ref{lem:spectral-to-semigroup}.
\end{proof}

In all that follows, we abbreviate the characteristic function of $Y_T$ by
\[\chi_{T,x}(u) = \bbE^{\bbPh_x ^{\beta,T}}[\exp(iuY_T)]\]
and take $\lambda_u =\lambda(\beta-iu)$ to be the analytic eigenvalue found in the previous proposition. For convenience we  abbreviate $\theta_\beta = \lambda' (\beta)$, as well as $\sigma^2 _\beta = - \lambda '' (\beta)$. The following lemma implements the heuristic that $Y_T$ should satisfy a CLT rigorously, and provides the remaining Fourier estimate to upgrade the CLT to a LLT.

\begin{lemma}
\label{lem:fourier-estimates}
Let $K\subset\mathbb R^d$ be compact.  There exist
$\varepsilon,c,C,\gamma>0$ such that, for all $T\geq1$,
\begin{equation}
 \sup_{x\in K}|\chi_{T,x}(u)|
 \leq C\bigl(\mathrm e^{-cTu^2}+ \mathrm e^{-\gamma T}\bigr),
 \qquad |u|\leq\varepsilon.
 \label{eq:small-frequency-bound}
\end{equation}
Moreover, locally uniformly in $(x,v)\in\mathbb R^d\times\mathbb R$,
\begin{equation}
 \chi_{T,x}\left(\frac v{\sqrt T}\right)
 \longrightarrow \mathrm e^{-\sigma_\beta^2v^2/2}.
 \label{eq:scaled-characteristic-limit}
\end{equation}
Finally, for every $0<\delta<R<\infty$, there exist
$C_{K,\delta,R},\eta_{\delta,R}>0$ such that
\begin{equation}
 \sup_{\substack{x\in K\\\delta\leq|u|\leq R}}
 |\chi_{T,x}(u)|
 \leq C_{K,\delta,R}\mathrm e^{-\eta_{\delta,R}T}.
 \label{eq:annular-characteristic-bound}
\end{equation}
\end{lemma}
\begin{proof}
Fix $\tau>0$ as in Lemma~\ref{lem:real-boundary-smoothing} and abbreviate $\lambda=\lambda_0=\lambda(\beta)$.  We first work with
$T\geq2\tau+1$; all estimates for $1\leq T<2\tau+1$ are then absorbed by
enlarging their constants.  For
$|u|\leq\varepsilon$, use Proposition \ref{prop:analytic-band}
and
\eqref{eq:local-spectral-decomposition} to obtain
\begin{equation}
 Z_T(\beta-iu,x)
 =\mathrm e^{-T\lambda_u}a_x(u)
  +O_K\bigl(\mathrm e^{-(\lambda+\gamma_1)T}\bigr),
 \label{eq:complex-partition-expansion}
\end{equation}
where, after changing the value of $\gamma_1>0$ if necessary,
 $a_x(u)
 = \mathrm e^{2\tau\lambda_u}\ell_{u,x}(P_ug_u)$.
Lemma~\ref{lem:real-boundary-smoothing} and
Proposition~\ref{prop:analytic-band} show that $a_x(u)$ is continuous in
$u$, locally uniformly in $x$; continuity in $x$ follows from the
Feynman-Kac heat-kernel representation.  At $u=0$, since
$P_0f=\phi_\beta\ip{f}{\phi_\beta}$,
\begin{align*}
 a_x(0)
 &=\mathrm  e^{2\tau\lambda}
   \bigl(S_0(\tau)P_0S_0(\tau)\mathbf 1\bigr)(x)
 =\phi_\beta(x)\ip{\mathbf 1}{\phi_\beta}
 =\phi_\beta(x)\norm{\phi_\beta}_1.
\end{align*}
In particular, $a_x(0)$ is bounded away from zero on $K$.
Put
\[
 \Psi(u)=\lambda_u-\lambda+iu\theta_\beta.
\]
Using 
\begin{equation}
 \chi_{T,x}(u)
 =\mathrm e^{-iu\theta_\beta T}
  \frac{Z_T(\beta-iu,x)}{Z_T(\beta,x)},
 \label{eq:characteristic}
\end{equation}
\eqref{eq:complex-partition-expansion}, and the real partition-function
asymptotic \eqref{eq:partition-asymptotic-rigorous}, we obtain
\begin{equation}
 \chi_{T,x}(u)
 =\frac{a_x(u)}{a_x(0)}\mathrm e^{-T\Psi(u)}
  +O_K(\mathrm e^{-\gamma_2T}),
 \qquad |u|\leq\varepsilon.
 \nonumber
\end{equation}
The amplitude is locally bounded, tends to one as $u\to0$, uniformly for
$x\in K$, and \eqref{eq:quadratic-real-part} gives
$\Re\Psi(u)\geq c_0u^2$.  This proves
\eqref{eq:small-frequency-bound}.  If $u=v/\sqrt T$, then
\[
 T\Psi(v/\sqrt T)
 =\frac12\sigma_\beta^2v^2+O(|v|^3T^{-1/2}),
\]
which proves \eqref{eq:scaled-characteristic-limit}.
For $\delta\leq|u|\leq R$, Lemma~\ref{lemma:strict_spectral_gap}, together with Lemma~\ref{lem:real-boundary-smoothing},
 gives
\[
 |Z_T(\beta-iu,x)|
 \leq C \mathrm e^{-(\lambda+\eta)(T-2\tau)}
\]
uniformly in $x\in K$ and $u$ in the annulus.  On the other hand,
\eqref{eq:partition-asymptotic-rigorous} and positivity of $\phi_\beta$ give
\[
 \inf_{x\in K}Z_T(\beta,x)
 \geq c_K \mathrm e^{-\lambda T}
\]
for all sufficiently large $T$.  Using this, the representation obtained in \eqref{eq:characteristic} and  \eqref{eq:partition-asymptotic-rigorous}
verifies
\eqref{eq:annular-characteristic-bound}.
\end{proof}
\begin{proof}[Proof of Theorem \ref{thm:llt}]
Apply Lemma~\ref{lem:stone-smoothing} to the laws of $Y_T$ under
$\hat\prob^{\beta,T}_x$.  Its three hypotheses are precisely
\eqref{eq:small-frequency-bound},\eqref{eq:scaled-characteristic-limit}
 and
\eqref{eq:annular-characteristic-bound}.
\end{proof}

\begin{proof}[Proof of Theorem \ref{thm:main}]
        Choose $\beta = \beta_\theta>0$ so that $\lambda'(\beta )=\theta$, which is possible thanks to Lemma \ref{lem:range-constraint-coupling}. We first verify \eqref{equ:precise_asymptotics}.
    As before, we find
\begin{align*}
 \prob_x(A_T\leq\theta T+z)
 &=Z_T(\beta,x)\mathrm e^{\beta\theta T}
   \bbE^{\bbPh_x ^{\beta,T}}
   \left[\mathrm e^{\beta Y_T}\mathbf 1_{\{Y_T\leq z\}}\right].
\end{align*}
Put
\[h(y)=\exp(\beta y)\mathbf 1_{(-\infty,0]}(y).
\]
This function is directly Riemann-integrable and
\[\exp(\beta Y_T)\mathbf 1_{{Y_T\leq z}}=\exp(\beta z)h(Y_T-z),\qquad\int_{\mathbb R} \de y \, h(y)=\frac{1}{\beta}.\]
Applying Theorem~\ref{thm:llt} with translation parameter  $a=z $ therefore gives
\[\sqrt T\,\mathbb E^{\hat{\bbP}_x ^{\beta,T}}\left[\exp(\beta Y_T)\mathbf 1_{{Y_T\leq z}}\right]\longrightarrow\frac{\exp(\beta z)}{\beta\sqrt{2\pi\sigma_\beta^2}},
\]
locally uniformly in  $(x,z) $.
Combining this with Proposition~\ref{prop:partition-asymptotic-rigorous} then yields
\begin{align*}
 \prob_x(A_T\leq\theta T+z)
 &\sim
 \mathrm e^{-\lambda(\beta)T}\phi_\beta(x)\norm{\phi_\beta}_1
 \mathrm e^{\beta\theta T}
 \frac{\mathrm e^{\beta z}}
      {\beta\sqrt{2\pi\sigma_\beta^2T}},
\end{align*}
which is precisely \eqref{equ:precise_asymptotics}. For fixed $S\geq0$, $a\in\mathbb R$, and $y\in\mathbb R^d$, we therefore find
\begin{equation}
    \nonumber
 \frac{\prob_y(A_{T-S}\leq\theta T-a)}
      {\prob_x(A_T\leq\theta T)}
 \longrightarrow
 \mathrm e^{\lambda(\beta)S-\beta a}
 \frac{\phi_\beta(y)}{\phi_\beta(x)}.
\end{equation}
This is seen by
\[
 \theta T-a=\theta(T-S)+(\theta S-a)
\]
and applying \eqref{equ:precise_asymptotics} to the numerator with the fixed shift
$z=\theta S-a$, and using  $I_Q(\theta)+\beta\theta=\lambda(\beta)$.
Finally,
for $F\in\cF_S$, the Markov property gives
\begin{align}
 \prob_x(F\mid A_T\leq\theta T)
 =
 \bbE^{\bbP_x}\left[
  \mathbf 1_F
  \frac{
   \prob_{W_S}\bigl(\Tilde A_{T-S}\leq\theta T-  A_S\bigr)}
  {\prob_x(A_T\leq\theta T)}
 \right].
 \label{eq:markov-ratio}
\end{align}
Here, $\Tilde A_{T-S}$ is random with respect to the inner probability  $\prob_{W_S}$.
For every fixed $\mathbb R^d$-valued Brownian path on $[0,S]$, taking
$y=W_S$ and $a=A_S$ shows that the density in
\eqref{eq:markov-ratio} converges to
\[
 M_S^\beta
 =\mathrm e^{\lambda(\beta)S-\beta A_S}
  \frac{\phi_\beta(W_S)}{\phi_\beta(x)}.
\]
The approximating densities have integral one, and
$\bbE^{\bbP_x}[M_S^\beta]=1$ by Proposition~\ref{prop:ground-state-transform-rigorous}.  Scheff\'e's Lemma gives
$L^1(\prob_x)$ convergence of the densities, which is exactly total-variation
convergence on $\cF_S$.  Tightness on each $C([0,S];\mathbb R^d)$, followed by the usual
projective argument, gives weak convergence on $C([0,\infty);\mathbb R^d)$.
\end{proof}
\begin{proof}[Proof of Corollary \ref{cor:fixed-width-window}]
For $z\in\mathbb R$, write
\[
    F_T(z)=\mathbb P_x(A_T\leq\theta T+z).
\]
Abbreviate $E_T ^{a,b} = \{\theta T+a<A_T\leq \theta T+b \}$. Since
$\mathbb P_x(E_T^{a,b})=F_T(b)-F_T(a)$,
the locally uniform asymptotic in Theorem~\ref{thm:main} gives
\begin{align*}
\mathbb P_x(E_T^{a,b})
&\sim
\frac{\phi_{\beta_\theta}(x)\|\phi_{\beta_\theta}\|_1}
     {\beta_\theta\sqrt{2\pi\sigma_\theta^2T}}\,
\mathrm e^{-I_Q(\theta)T}
\bigl(\mathrm e^{\,\beta_\theta b}-\mathrm e^{\,\beta_\theta a}\bigr) 
=
\frac{\phi_{\beta_\theta}(x)\|\phi_{\beta_\theta}\|_1}
     {\sqrt{2\pi\sigma_\theta^2T}}\,
\mathrm e^{-I_Q(\theta)T}
\int_a^b \de y \, \mathrm e^{\, \beta_\theta y}.
\end{align*}
The same argument is locally uniform in $(x,a,b)$ as long as $a<b$.

It remains to prove the assertion about the conditional law.  Fix
$S<\infty$.  By the Markov property, the density on $\mathcal F_S$ of
the conditioned law with respect to $\mathbb P_x$ is
\[
D_{T,S}^{a,b}
=
\frac{
\mathbb P_{W_S}\!\left(
\theta(T-S)+a+\theta S-A_S<A_{T-S}
\leq\theta(T-S)+b+\theta S-A_S
\right)}
{\mathbb P_x(E_T^{a,b})}.
\]
For every fixed $\mathbb R^d$-valued Brownian path on $[0,S]$, the starting point $W_S$ and
the shifts
\[
    a+\theta S-A_S,
    \qquad
    b+\theta S-A_S
\]
are fixed real numbers.  Applying the first part of the Corollary with
time horizon $T-S$ therefore gives
\begin{align*}
D_{T,S}^{a,b}
&\longrightarrow
\mathrm e^{I_Q(\theta)S+\beta_\theta(\theta S-A_S)}
\frac{\phi_{\beta_\theta}(W_S)}
     {\phi_{\beta_\theta}(x)}  
=
\mathrm e^{\lambda(\beta_\theta)S-\beta_\theta A_S}
\frac{\phi_{\beta_\theta}(W_S)}
     {\phi_{\beta_\theta}(x)}
=M_S,
\end{align*}
where we used
    $I_Q(\theta)+\beta_\theta\theta
    =\lambda(\beta_\theta)$.
The random variable $M_S$ is precisely the ground-state-transform
density on $\mathcal F_S$; see also \eqref{eq:ground-state-sde-rigorous}.  Moreover, $\mathbb E^{\bbP_x}[M_S]=1$.
Since also $\mathbb E^{\bbP_x}[D_{T,S}^{a,b}]=1$, Scheff\'e's Lemma yields
\[
    D_{T,S}^{a,b}\longrightarrow M_S
    \qquad\text{in }L^1(\mathbb P_x),
\]
which is equivalent to the claimed total-variation convergence.
\end{proof}

\section{Extension to General Reference Markov Processes}
\label{sec:examples}

Although Theorem~\ref{thm:main} is formulated for Brownian motion, its
proof uses only a small number of structural properties of the reference
process.  We briefly isolate these properties before discussing concrete
examples which also verify these properties.
Let $X=(X_t)_{t\geq 0}$ be an $\R^d$-valued conservative Markov process which is reversible
with respect to a measure $\mu$, and let $L$ denote its self-adjoint
generator on $L^2(\mu)$.  For a nonnegative, nonconstant observable $Q$,
consider
\[
    A_T=\int_0^T \de s \, Q(X_s)
\]
and the Feynman-Kac operators
\[
    \mathsf H_z=-L+zQ,
    \qquad \Re z>0.
\]
The argument used for Brownian motion extends whenever the following
ingredients are available.

\textbf{(I)} First, the operators $\mathsf H_z$ should be defined by a holomorphic family
of closed sectorial forms with a common form domain.  For every $\beta>0$,
the real operator $\mathsf H_\beta$ should have compact resolvent and a
positivity-improving semigroup.  Its lowest eigenvalue
$\lambda(\beta)$ is then isolated and simple, with a strictly positive
normalized ground state $\psi_\beta$, and analytic perturbation theory
applies near every positive $\beta$. We assume the ground state to be continuous.

\textbf{(II)} Second, one needs spectral separation of the nonzero Fourier modes.  More
precisely, for every $u\neq0$,
\[
    \inf\bigl\{\Re\zeta:
        \zeta\in\spec(\mathsf H_{\beta-iu})\bigr\}
       >\lambda(\beta),
\]
and the gap must be uniform when $u$ ranges over a compact annulus
$\delta\leq |u|\leq R$.  For reversible diffusions and reversible Markov
chains, this follows from the Markovian inequality
\[
    \mathcal E(|f|,|f|)\leq \mathcal E(f,f)
\]
for the Dirichlet form $\mathcal E$ of $-L$, together with rigidity in the
equality case.  Indeed, equality at the ground-state energy forces
$|f|=\psi_\beta$ and a constant phase on the irreducible state space; the
eigenvalue equation for $\mathsf H_{\beta-iu}$ would then force $Q$ to be
constant.  Compactness of the form-domain embedding (due to the compact resolvent of $\mathsf H_\beta$) makes the resulting
strict gap uniform on compact frequency annuli.

\textbf{(III)} Third, the spectral estimates obtained in $L^2(\mu)$ must be transferable
to a process started from a fixed point.  It suffices to have a locally
uniform smoothing estimate of the form
\[
    \sup_{x\in D}|P_\tau f(x)|
       \leq C_{D,\tau}\|f\|_{L^2(\mu)}
\]
for every compact set $D$ and every $\tau>0$.  When $\mu$ is a probability
measure, one also has $\mathbf 1\in L^2(\mu)$ and therefore directly obtains
\[
    Z_T(z,x)
      =\mathbb E^{\bbP_x}\!\left[\mathrm e^{-zA_T}\right]
      =\bigl(\mathrm e^{-T\mathsf H_z}\mathbf 1\bigr)(x).
\]
This replaces the additional endpoint-smoothing construction needed in the
Brownian setting, where Lebesgue measure is infinite and
$\mathbf 1\notin L^2(\mathbb R^d)$.

\textbf{(IV)} Fourth, one must identify the range of $\lambda'$ and verify that
\[
    v_\beta=-\lambda''(\beta)>0.
\]
For every constraint value $\theta$ in this range, there is then a unique
$\beta_\theta>0$ satisfying $\lambda'(\beta_\theta)=\theta$.  

\textbf{(V)} Finally, we require that the ground-state transform is conservative,
or equivalently that
\[
    M_t^\beta
      =\mathrm e^{\lambda(\beta)t-\beta A_t}
        \frac{\psi_\beta(X_t)}{\psi_\beta(x)}
\]
is a mean-one martingale under $\bbP_x$.  This defines the transformed
path measure $\bbP_x^{(\beta)}$ on every $\mathcal F_t$ by
\[
    \frac{\de\bbP_x^{(\beta)}}{\de\bbP_x}\bigg|_{\mathcal F_t}
       =M_t^\beta.
\]
Conditions \textbf{(I)}-\textbf{(IV)} yield the corresponding local
limit theorem and sharp lower-tail asymptotic.  Condition \textbf{(V)}
then allows the Markov property and Scheff\'e arguments to identify the
conditional path limit on every fixed time interval as the conservative
ground-state transform generated by
\[
    L^{(\beta_\theta)}f
      =
      \psi_{\beta_\theta}^{-1}
      \bigl(\lambda(\beta_\theta)-\mathsf H_{\beta_\theta}\bigr)
      \bigl(\psi_{\beta_\theta}f\bigr).
\]

\begin{theorem}
\label{prop:reversible-extension-principle}
Let $X$ be an $\R^d$-valued conservative Markov process reversible with respect to a
probability measure $\mu$, and let $Q$ be a nonnegative, nonconstant
observable.  Suppose that conditions \textbf{(I)}-\textbf{(V)} above
hold.  Let
\[
    \theta\in\lambda'((0,\infty)),
    \qquad
    \lambda'(\beta_\theta)=\theta,
\]
and put
\[
    v_\theta=-\lambda''(\beta_\theta)>0,
    \qquad
    I_Q(\theta)
       =\lambda(\beta_\theta)-\beta_\theta\theta.
\]
Then, for every fixed initial state $x$ and every fixed $z\in\mathbb R$,
\[
    \bbP_x(A_T\leq\theta T+z)
    \sim
    \frac{
      \psi_{\beta_\theta}(x)
      \langle\psi_{\beta_\theta},\mathbf 1\rangle_{L^2(\mu)}
    }{
      \beta_\theta\sqrt{2\pi v_\theta T}
    }
    \exp\!\left(
       -I_Q(\theta)T+\beta_\theta z
    \right).
\]
Whenever the estimates in \textbf{(III)} are locally uniform in the
initial state, this asymptotic is locally uniform in $(x,z)$.

Moreover, for every $S<\infty$,
\[
    \mathcal L_{\bbP_x}
       \bigl((X_t)_{0\leq t\leq S}\mid A_T\leq\theta T\bigr)
    \longrightarrow
    \mathcal L_{\bbP_x^{(\beta_\theta)}}
       \bigl((X_t)_{0\leq t\leq S}\bigr)
\]
in total variation.  The corresponding fixed-width conclusion of
Corollary~\ref{cor:fixed-width-window} holds as well.
\end{theorem}

\begin{proof}
Condition \textbf{(I)} gives the analytic ground-state expansion near
zero frequency, while \textbf{(II)} gives exponential contraction on
every compact frequency annulus.  Condition \textbf{(III)} transfers
these $L^2(\mu)$ estimates to a process started from a fixed state.
Consequently, the proof of Theorem~\ref{thm:llt} applies verbatim and
gives the corresponding local limit theorem.

Condition \textbf{(IV)} identifies the unique tilting parameter
$\beta_\theta$ and guarantees that the limiting variance is strictly
positive.  The untilting calculation in the proof of
Theorem~\ref{thm:main} then gives the displayed sharp asymptotic.
Finally, \textbf{(V)}, the Markov property and Scheff\'e arguments used there identify
the limiting density on $\mathcal F_S$ as
\[
    M_S^{\beta_\theta}
      =\mathrm e^{\lambda(\beta_\theta)S-\beta_\theta A_S}
        \frac{\psi_{\beta_\theta}(X_S)}
             {\psi_{\beta_\theta}(x)}.
\]
This proves the total-variation convergence.  Taking differences of the
locally uniform tail asymptotics proves the fixed-width assertion.
\end{proof}

We now verify the provided
criteria in three concrete examples. In all examples, condition \textbf{(V)} is fulfilled as it simply follows from the eigenvalue equation
\[(\mathrm e^{-tH_\beta}\psi_\beta)(x) = \mathrm e^{-t\lambda(\beta)}\psi_\beta(x).\]
Indeed, this identity implies that $M_t ^\beta$ has mean one, and the Markov property then shows that $(M_t ^\beta)_{t \ge 0}$ is a martingale.

\subsection{Reversible Ornstein-Uhlenbeck processes}
\label{rem:ou-extension}
Let $K$ be a symmetric positive-definite $d\times d$ matrix and let
$\sigma>0$.  Consider
\[
    \de X_t=-KX_t\,\de t+\sigma\,\de B_t,
    \qquad X_0=x\in\mathbb R^d.
\]
Denote the corresponding path measure by $\bbP_x^{\mathrm{OU}}$.  This
Ornstein-Uhlenbeck process is reversible with respect to
\[
    \mu_K(\de x)
      =\frac{(\det K)^{1/2}}{(\pi\sigma^2)^{d/2}}
        \exp\!\left(-\frac{x^{\mathsf T}Kx}{\sigma^2}\right)\de x,
\]
and its generator is
\[
    L_K=\frac{\sigma^2}{2}\Delta-(Kx)\cdot\nabla.
\]
For a continuous, nonnegative, and nonconstant observable $Q$, replace
$H_z$ by
\[
    K_z=-L_K+zQ
    \qquad\text{on }L^2(\mu_K).
\]

The Ornstein-Uhlenbeck process satisfies the structural requirements
described above.  An important difference from Brownian motion is that the
reference dynamics are already confining.  In particular, the unperturbed
operator $-L_K$ has compact resolvent, so bounded nonconstant observables
are permitted as well.  This can also be seen through a unitary
transformation.  If $\rho_K$ is the density of $\mu_K$ and
$Uf=\rho_K^{1/2}f$, then
\[
    UK_zU^{-1}
      =-\frac{\sigma^2}{2}\Delta
        +\frac{x^{\mathsf T}K^2x}{2\sigma^2}
        -\frac12\operatorname{tr}K+zQ(x).
\]
Thus the OU drift itself supplies the confining quadratic potential.

Let $\lambda(\beta)$ be the lowest eigenvalue of $K_\beta$, and let
$\psi_\beta>0$ be its normalized ground state in $L^2(\mu_K)$.  For
\[
    \theta\in\Theta_Q^{\mathrm{OU}}
       = \{ \lambda'(\beta): \beta \in (0,\infty) \},
    \qquad
    \lambda'(\beta_\theta)=\theta,
\]
we now spell out the precise replacements required in the Brownian proof.

\begin{enumerate}
\item[(i)] \emph{Weighted spectral problem.}
Lebesgue measure and the Brownian Dirichlet form are replaced by
$\mu_K$ and
\[
    \mathcal E_K(f,g)
      =\frac{\sigma^2}{2}
        \int_{\mathbb R^d}\mu_K (\de x) \, \nabla f (x) \cdot\nabla\overline g (x),
\]
with common form domain
\[
    \mathcal D_{K,Q}
      =W^{1,2}(\mu_K)\cap
        L^2(\mathbb R^d,Q(x)\, \mu_K(\de x)).
\]
The Hermite decomposition of $-L_K$ gives compact resolvent; hence the
embedding $\mathcal D_{K,Q}\hookrightarrow L^2(\mu_K)$ is compact as well.
The Mehler kernel is strictly positive \cite{Stroock20,Bogachev}, so the Feynman-Kac semigroup is
positivity improving.  Consequently, the lowest eigenvalue is simple, and
the form family $K_z=-L_K+zQ$ is holomorphic on $\Re z>0$.  The
Hellmann-Feynman and reduced-resolvent arguments used for
Propositions~\ref{prop:ground-state-spectral-theory} and
\ref{prop:hellmann-feynman} therefore give
\[
    \lambda'(\beta)=\int \mu_K(\de x)\,  Q(x)\psi_\beta^2(x),
    \qquad
    \lambda''(\beta)<0.
\]

\item[(ii)] \emph{Range of the constraint.}
Lemma~\ref{lem:range-constraint-coupling} is specific to Brownian motion and
requires a replacement.  Since $\mu_K$ is a probability measure, the unperturbed ground state is the constant function $\mathbf 1$, and $\lambda'(0+)$ is no longer automatically infinite.  For the quadratic observable below, perturbation at $\beta=0$ gives
\[
    \lambda'(0+)=\int_{\mathbb R^d}\mu_K(\de x) \, Q(x).
\]
Thus the OU extension concerns lower deviations, with
$\theta\in\Theta_Q^{\mathrm{OU}}=\lambda'((0,\infty))$.  For the quadratic
case this becomes the explicit interval displayed below.

\item[(iii)] \emph{Nonzero Fourier modes.}
In Lemma~\ref{lemma:strict_spectral_gap}, replace $H_{\beta-iu}$ by
$K_{\beta-iu}$ and use once again the diamagnetic inequality to obtain
\[
    \mathcal E_K(|f|,|f|)\leq\mathcal E_K(f,f).
\]
Equality forces $|f|=\psi_\beta$ and a constant phase on the connected state
space $\mathbb R^d$; the eigenvalue equation would then force $Q$ to be
constant.  This proves the strict gap for $u\neq0$.  Compactness of the
weighted form-domain embedding gives the uniform gap on every annulus
$\delta\leq|u|\leq R$.  The analytic small-frequency expansion is then the
same one used in Lemma~\ref{lem:fourier-estimates}.

\item[(iv)] \emph{A process started from a point.}
Here $\mathbf 1\in L^2(\mu_K)$, so the smoothing construction in
Lemma~\ref{lem:real-boundary-smoothing} is unnecessary.  Instead, we directly find
\[
    Z_T(z,x)=\mathbb E^{\bbP_x ^\mathrm{OU}}[\mathrm e^{-zA_T}]
            =(\mathrm e^{-TK_z}\mathbf 1)(x).
\]
If $P_t^K$ denotes the OU semigroup, its Mehler kernel satisfies, for every
compact $D\subset\mathbb R^d$ and $\tau>0$,
\[
    \sup_{x\in D}|P_\tau^Kf(x)|
       \leq C_{D,\tau}\|f\|_{L^2(\mu_K)};
\]
see for example \cite{Stroock20}.
Applying this estimate after the $L^2(\mu_K)$ spectral expansion yields,
locally uniformly in $x$,
\[
    Z_T(\beta,x)
      =\mathrm e^{-\lambda(\beta)T}\psi_\beta(x)
       \langle\psi_\beta,\mathbf 1\rangle_{L^2(\mu_K)}
       \bigl(1+O(\mathrm e^{-cT})\bigr).
\]
The same smoothing estimate, combined with complex Feynman-Kac domination,
passes the nonzero-frequency $L^2$ bounds to fixed starting points.

\end{enumerate}

\begin{corollary}
\label{cor:ou-extension}
Let $K$ be symmetric positive definite, let $\sigma>0$, and let
$Q:\mathbb R^d\to[0,\infty)$ be continuous and nonconstant.  Define
\[
    A_T=\int_0^T \de s \, Q(X_s)
\]
for the Ornstein-Uhlenbeck process
\[
    \de X_t=-KX_t\,\de t+\sigma\,\de B_t,
    \qquad X_0=x.
\]
Let $\lambda(\beta)$ and $\psi_\beta$ be the lowest eigenvalue and
normalized positive ground state of
\[
    K_\beta=-L_K+\beta Q
    \qquad\text{on }L^2(\mu_K).
\]
For $\theta\in\Theta_Q^{\mathrm{OU}}
       =\lambda'((0,\infty))$ and $\lambda'(\beta_\theta)=\theta$,
put $v_\theta=-\lambda''(\beta_\theta)>0$ as well as 
\[
    I_Q^{\mathrm{OU}}(\theta)
       =\lambda(\beta_\theta)-\beta_\theta\theta.
\]
Then, for every fixed $z\in\mathbb R$,
\[
    \bbP_x^{\mathrm{OU}}(A_T\leq\theta T+z)
    \sim
    \frac{
      \psi_{\beta_\theta}(x)
      \langle\psi_{\beta_\theta},\mathbf 1\rangle_{L^2(\mu_K)}
    }{
      \beta_\theta\sqrt{2\pi v_\theta T}
    }
    \exp\!\left(
       -I_Q^{\mathrm{OU}}(\theta)T+\beta_\theta z
    \right).
\]
The asymptotic is locally uniform in $(x,z)$.

Moreover, on every fixed time interval, the conditioned process converges
in total variation to the ground-state diffusion
\[
    \de Y_t
      =
      \left[
        -KY_t+\sigma^2\nabla\log
        \psi_{\beta_\theta}(Y_t)
      \right]\de t
      +\sigma\,\de B_t,
    \qquad Y_0=x.
\]
\end{corollary}

\begin{proof}
The verification of conditions \textbf{(I)}-\textbf{(IV)} was given
above.  The result therefore follows from
Theorem~\ref{prop:reversible-extension-principle}.
\end{proof}
The quadratic case yields a rather explicit and very nice form; see also
\cite{DelMoralHorton} for a substantially broader Riccati formulation for
linear diffusions.  Let
$Q(x)=x^{\mathsf T}Cx$, where $C$ is symmetric positive definite, and set
$R_\beta=(K^2+2\beta\sigma^2C)^{1/2}$, where the positive-definite square root is meant.  Direct substitution in
the transformed oscillator gives
\[
    \lambda(\beta)
      =\frac12\operatorname{tr}(R_\beta-K),
    \qquad
    \lambda'(\beta)
      =\frac{\sigma^2}{2}\operatorname{tr}(R_\beta^{-1}C),
\]
and
\[
    \psi_\beta(x)
      =\left(\frac{\det R_\beta}{\det K}\right)^{1/4}
       \exp\!\left(
          -\frac{x^{\mathsf T}(R_\beta-K)x}{2\sigma^2}
       \right).
\]
Consequently,
\[
\psi_\beta(x)\langle\psi_\beta,\mathbf 1\rangle_{L^2(\mu_K)}
=\left(\frac{\det(2R_\beta)}{\det(R_\beta+K)}\right)^{1/2}
 \exp\!\left(
    -\frac{x^{\mathsf T}(R_\beta-K)x}{2\sigma^2}
 \right),
\]
and the limiting process is the Ornstein-Uhlenbeck process
\[
    \de Y_t=-R_{\beta_\theta}Y_t\,\de t+\sigma\,\de B_t.
\]
Its invariant covariance is $\sigma^2R_{\beta_\theta}^{-1}/2$, and hence
the invariant mean of $Q$ is $\frac{\sigma^2}{2}
    \operatorname{tr}(CR_{\beta_\theta}^{-1})=\theta$.
If $K$ and $C$ commute, choose coordinates in which
$K=\operatorname{diag}(k_1,\ldots,k_d)$ and
$C=\operatorname{diag}(c_1,\ldots,c_d)$.  With
\[
    r_j(\beta)=\sqrt{k_j^2+2\beta\sigma^2c_j},
\]
one has
\[
    \lambda(\beta)=\frac12\sum_{j=1}^d(r_j(\beta)-k_j),
    \quad
    \lambda'(\beta)=\frac{\sigma^2}{2}
       \sum_{j=1}^d\frac{c_j}{r_j(\beta)},
    \quad
    -\lambda''(\beta)=\frac{\sigma^4}{2}
       \sum_{j=1}^d\frac{c_j^2}{r_j(\beta)^3}.
\]
Thus $\beta_\theta$ is the unique positive solution of the middle equation
for $0<\theta<\frac{\sigma^2}{2} \operatorname{tr}(CK^{-1})$,
and all constants in the sharp tail are explicit up to this scalar equation.
For the isotropic choice $K=k I_d$ and $C=I_d$, the scalar equation can
also be solved explicitly.  Writing $r_\theta=\frac{d\sigma^2}{2\theta}$ and 
$0<\theta<\frac{d\sigma^2}{2k}$, 
we obtain
\[
    \beta_\theta
      =\frac{d^2\sigma^2}{8\theta^2}
       -\frac{k^2}{2\sigma^2}.
\]
Moreover,
\[
\psi_{\beta_\theta}(x)
\langle\psi_{\beta_\theta},\mathbf 1\rangle_{L^2(\mu_K)}
=\left(\frac{2r_\theta}{r_\theta+k}\right)^{d/2}
 \exp\!\left(-\frac{(r_\theta-k)|x|^2}{2\sigma^2}\right),
\]
and the limiting process solves
\[
    \de Y_t=-\frac{d\sigma^2}{2\theta}Y_t\,\de t+\sigma\,\de B_t.
\]
Its invariant law is centered Gaussian and satisfies
$\mathbb E [|Y|^2]=\theta$, exactly matching the imposed energy density.

\subsection{Cox-Ingersoll-Ross processes}
\label{ex:cir-reference}
A non-Gaussian one-dimensional example is provided by the
Cox-Ingersoll-Ross (CIR) process \cite{CIR}
\[
    \de X_t=\kappa(\eta-X_t)\,\de t
      +\sigma\sqrt{X_t}\,\de B_t,
    \qquad X_0=x>0,
\]
where $\kappa,\eta,\sigma>0$ and $2\kappa\eta\geq\sigma^2$.  Its generator is
\[
    L_{\mathrm{CIR}}f(x)
      =\frac{\sigma^2x}{2}f''(x)+\kappa(\eta-x)f'(x),
\]
and its reversible invariant law is the Gamma distribution
\[
    \mu_{\mathrm{CIR}}(\de x)
      =\frac{\delta_\kappa^\alpha}{\Gamma(\alpha)}
        x^{\alpha-1}\mathrm e^{-\delta_\kappa x}\,\de x,
    \qquad
    \alpha=\frac{2\kappa\eta}{\sigma^2},
    \quad
    \delta_\kappa=\frac{2\kappa}{\sigma^2}.
\]
For completeness, the proof justification here is the weighted analogue of
the OU discussion above.  The Dirichlet form is
\[
    \mathcal E_{\mathrm{CIR}}(f,g)
      =\frac{\sigma^2}{2}\int_0^\infty
        \mu_{\mathrm{CIR}}(\de x)\,x f'(x)\overline{g'(x)}.
\]
The generalized Laguerre polynomials diagonalize $-L_{\mathrm{CIR}}$, with
spectrum $\{n\kappa:n\in\mathbb N_0\}$, so the weighted form-domain
embedding is compact; the CIR transition kernel supplies the locally uniform
$L^2(\mu_{\mathrm{CIR}})$-to-pointwise smoothing estimate; see
\cite[Theorem~6.1]{MendozaLinetsky}.  The Feller condition fixes the
conservative boundary realization at $0$.  
Hence the spectral-gap, Fourier,
untilting, and ground-state-transform steps above apply. 
The range lemma is again replaced: here
$\lambda'(0+)=\int_0 ^\infty \mu_{\mathrm{CIR}}(\de x)\, x=\eta$, which explains the lower-deviation range $0<\theta<\eta$ below.

We specialize to  $A_T=\int_0^T\de s \,  X_s$ and take $r_\beta=\sqrt{\kappa^2+2\beta\sigma^2}$ together with $c_\beta=\frac{r_\beta-\kappa}{\sigma^2}$.
The normalized positive ground state of
$-L_{\mathrm{CIR}}+\beta x$ in $L^2(\mu_{\mathrm{CIR}})$ and the corresponding
ground energy are
\[
    \psi_\beta(x)
      =\left(\frac{r_\beta}{\kappa}\right)^{\alpha/2}
        \mathrm e^{-c_\beta x},
    \qquad
    \lambda(\beta)=\kappa\eta c_\beta.
\]
Consequently $\lambda'(\beta)=\frac{\kappa\eta}{r_\beta}$ and $-\lambda''(\beta)=\frac{\kappa\eta\sigma^2}{r_\beta^3}$.
For $0<\theta<\eta$, the equation $\lambda'(\beta_\theta)=\theta$ gives
\[
    r_{\beta_\theta}=\frac{\kappa\eta}{\theta},
    \qquad
    \beta_\theta
      =\frac{\kappa^2}{2\sigma^2}
        \left(\frac{\eta^2}{\theta^2}-1\right),
\]
    and hence
\[
    I_x^{\mathrm{CIR}}(\theta)
      =\frac{\kappa^2(\eta-\theta)^2}{2\sigma^2\theta},
    \qquad
    v_\theta = -\lambda''(\beta_\theta)
      =\frac{\sigma^2\theta^3}{\kappa^2\eta^2}.
\]
The endpoint factor in the sharp asymptotic is explicit and evaluates to
\[
    \psi_{\beta_\theta}(x)
    \langle\psi_{\beta_\theta},\mathbf 1\rangle_{L^2(\mu_{\mathrm{CIR}})}
      =\left(\frac{2r_{\beta_\theta}}
                    {\kappa+r_{\beta_\theta}}\right)^\alpha
        \exp\!\left(-\frac{(r_{\beta_\theta}-\kappa)x}{\sigma^2}\right).
\]

\begin{corollary}
\label{cor:cir-extension}
Let $X$ be the CIR process above, take
$A_T=\int_0^T \de s \, X_s$
and fix $0<\theta<\eta$.  With $r_{\beta_\theta}=\frac{\kappa\eta}{\theta}$ and $\beta_\theta
      =\frac{\kappa^2}{2\sigma^2}
        \left(\frac{\eta^2}{\theta^2}-1\right)$,
put
\[
    I_x^{\mathrm{CIR}}(\theta)
      =\frac{\kappa^2(\eta-\theta)^2}{2\sigma^2\theta},
    \qquad
    v_\theta
      =\frac{\sigma^2\theta^3}{\kappa^2\eta^2}.
\]
Then, for every fixed $x>0$ and $z\in\mathbb R$,
\[
\begin{split}
    \bbP_x^{\mathrm{CIR}}(A_T\leq\theta T+z)
    \sim{}&
    \frac{1}{\beta_\theta\sqrt{2\pi v_\theta T}}
    \left(
       \frac{2r_{\beta_\theta}}
            {\kappa+r_{\beta_\theta}}
    \right)^\alpha
    \exp\!\left(
       -\frac{(r_{\beta_\theta}-\kappa)x}{\sigma^2}
    \right)
    \exp\!\left(
       -I_x^{\mathrm{CIR}}(\theta)T+\beta_\theta z
    \right).
\end{split}
\]
On every fixed time interval, the conditioned process converges in total
variation to the CIR process
\[
    \de Y_t
      =r_{\beta_\theta}(\theta-Y_t)\,\de t
       +\sigma\sqrt{Y_t}\,\de B_t,
    \qquad Y_0=x.
\]
Its invariant law is Gamma with shape $\alpha$ and mean $\theta$.
\end{corollary}

\begin{proof}
The spectral and smoothing properties verified above give
conditions \textbf{(I)}-\textbf{(III)}, while
$\lambda'((0,\infty))=(0,\eta)$ gives condition \textbf{(IV)}.
Theorem~\ref{prop:reversible-extension-principle} therefore applies.
The displayed constants and the transformed drift follow from the
preceding calculations.
\end{proof}

\subsection{A birth-death process with immigration}
Let us now take as reference path measure a process that is not a diffusion. Let
$X=(X_t)_{t\geq 0}$ be the continuous-time Markov chain on
$\Z_{\geq 0}$ with transition rates
\[
 q(n,n+1)=a+bn,
 \qquad
 q(n,n-1)=cn,
 \qquad
 a,b,c>0,\quad c>b
\]
and generator $L$.
It is reversible with respect to the negative-binomial probability
measure
\[
 \pi(n)
 =
 \left(1-\frac bc\right)^{a/b}
 \frac{\Gamma(n+a/b)}{\Gamma(a/b)n!}
 \left(\frac bc\right)^n,
 \qquad n\in\Z_{\geq 0}.
\]
As before, $A_T=\int_0^T \de s \, X_s$.
For $\beta>0$, let $r_\beta\in(0,1)$ be the smaller solution of
\[
 br_\beta+\frac{c}{r_\beta}=b+c+\beta.
\]
Direct substitution shows that the normalized ground state of
$H_\beta=-L+\beta n$ on $\ell^2(\pi)$
is of the form $\psi_\beta(n)=C_\beta r_\beta^n$, and that the
corresponding ground-state energy satisfies
\[
 \lambda(\beta)=a(1-r_\beta),
 \qquad
 \lambda'(\beta)=\frac{a r_\beta^2}{c-br_\beta^2}.
\]

Let us briefly explain why the preceding argument applies. The
Meixner-polynomial spectral decomposition of $-L$ has eigenvalues $k(c-b)$, $k\in\Z_{\geq 0}$; see for example \cite[Section~2]{AscioneLeonenkoPirozzi}. Therefore $H_\beta$ has compact
resolvent. Moreover, the graph Dirichlet form satisfies
\[
 \frac12 \sum_{n \ge 0}\sum_{m\neq n}\pi(n)q(n,m)
 \bigl(|f(m)|-|f(n)|\bigr)^2
 \leq
 \frac12\sum_{n \ge 0}\sum_{m\neq n}\pi(n)q(n,m)|f(m)-f(n)|^2.
\]
Equality forces the phase of $f$ to be constant on the connected
graph $\Z_{\geq 0}$. Consequently, an eigenfunction of
$H_{\beta-iu}$ whose eigenvalue has real part $\lambda(\beta)$ would
force the observable $n\mapsto n$ to be constant. Thus every
nonzero Fourier mode has a strict spectral gap. Compactness makes
this gap uniform when $u$ ranges over a compact annulus, and the
local limit and untilting arguments remain unchanged. Notice in
particular that, despite the discrete state space, $A_T$ is
nonlattice because the holding times are continuously distributed.

For $0<\theta<a/(c-b)$, put
\[
 r_\theta=\sqrt{\frac{c\theta}{a+b\theta}},
 \qquad
 \beta_\theta
 =
 br_\theta+\frac{c}{r_\theta}-b-c,
 \qquad
 v_\theta = -\lambda''(\beta_\theta)
 =
 \frac{2ac r_\theta^3}{(c-br_\theta^2)^3}.
\]
Then $\lambda'(\beta_\theta)=\theta$,
$v_\theta=-\lambda''(\beta_\theta)>0$, and the corresponding rate
function is
\[
 I_{\rm BDI}(\theta)
 =
 \lambda(\beta_\theta)-\beta_\theta\theta
 =
 \left(\sqrt{a+b\theta}-\sqrt{c\theta}\right)^2.
\]

\begin{corollary}
\label{cor:bdi-extension}
Consider the birth-death process with immigration introduced above and let
    $A_T=\int_0^T \de s \, X_s$.
For $0<\theta<a/(c-b)$, define $r_\theta=\sqrt{\frac{c\theta}{a+b\theta}}$ and $\beta_\theta
      =br_\theta+\frac{c}{r_\theta}-b-c$.
Moreover, let
\[
    v_\theta
      =\frac{2ac r_\theta^3}{(c-br_\theta^2)^3},
    \qquad
    I_{\mathrm{BDI}}(\theta)
      =\left(
          \sqrt{a+b\theta}-\sqrt{c\theta}
       \right)^2.
\]
Then, for every fixed $n\in\mathbb Z_{\geq0}$ and $z\in\mathbb R$,
\[
    \bbP_n(A_T\leq\theta T+z)
    \sim
    \frac{
      \psi_{\beta_\theta}(n)
      \langle\psi_{\beta_\theta},\mathbf 1\rangle_{\ell^2(\pi)}
    }{
      \beta_\theta\sqrt{2\pi v_\theta T}
    }
    \exp\!\left(
       -I_{\mathrm{BDI}}(\theta)T+\beta_\theta z
    \right).
\]
On every fixed time interval, the conditioned process converges in total
variation to the birth-death process with rates
\[
    q^{(\theta)}(n,n+1)=r_\theta(a+bn),
    \qquad
    q^{(\theta)}(n,n-1)=\frac{c}{r_\theta}n.
\]
Its invariant law is negative binomial with shape parameter $a/b$ and
mean $\theta$.
\end{corollary}

\begin{proof}
The preceding discussion verifies conditions
\textbf{(I)}-\textbf{(IV)}, so
Theorem~\ref{prop:reversible-extension-principle} applies.
The formulas for the transformed rates follow from
$\psi_{\beta_\theta}(n)\propto r_\theta^n$.  The resulting
negative-binomial invariant distribution has the same shape parameter as
$\pi$ and mean $\theta$.
\end{proof}

\subsection*{Acknowledgements}
Thanks to Frank Aurzada for making us aware of their conjecture. Thanks also to Volker Betz for useful discussions. The author also gratefully acknowledges support from DFG grant No 535662048.

\begin{appendix}

\section{Sectorial forms, complex Feynman-Kac semigroups, and Riesz projections}
\label{sec:sectorial-background}

This Appendix collects the functional-analytic facts used in the complex
Fourier argument.  We include the details because the complex operator
\[
        H_{\beta-iu}=-\frac12\Delta+(\beta-iu)Q
\]
is not self-adjoint when $u\neq0$, and hence the usual self-adjoint spectral
theorem is not directly available.  The appropriate replacement consists of
closed sectorial forms, holomorphic semigroups, and Riesz projections.  We
also give a proof of the complex Feynman-Kac formula under the assumptions
used in the paper.
Throughout this Appendix, the scalar product on a complex Hilbert space is
linear in its first variable.  Thus, on $L^2(\mathbb R^d)$,
\[
        \langle f,g\rangle=\int_{\mathbb R^d}\de x \,f(x)\overline{g(x)}.
\]
\subsection{From energy forms to operators}
\label{subsec:forms-to-operators}

Let $\mathcal H$ be a complex Hilbert space.  A sesquilinear form
$\mathfrak a$ is a map
\[
  \mathfrak a:\mathcal D(\mathfrak a)\times\mathcal D(\mathfrak a)
  \longrightarrow\mathbb C
\]
which is linear in the first variable and conjugate-linear in the second.  We
write $\mathfrak a[f]=\mathfrak a[f,f]$.  The form is densely defined if
$\mathcal D(\mathfrak a)$ is dense in $\mathcal H$.

For $0\leq\alpha<\pi/2$, let
$\overline{\Sigma_\alpha}
  =\{re^{i\vartheta}:r\geq0,\ |\vartheta|\leq\alpha\}$.
\begin{definition}
\label{def:sectorial-form}
A densely defined form $\mathfrak a$ is sectorial with vertex
$\omega\in\mathbb R$ and semi-angle $\alpha<\pi/2$ if
\[
  \mathfrak a[f]-\omega\|f\|^2\in\overline{\Sigma_\alpha},
  \qquad f\in\mathcal D(\mathfrak a).
\]
Equivalently, after a real shift, its imaginary part is controlled by its real
part: there is $C<\infty$ such that
\[
  |\Im \mathfrak a[f]|
  \leq C\bigl(\Re\mathfrak a[f]-\omega\|f\|^2\bigr),
  \qquad f\in\mathcal D(\mathfrak a).
\]
If $\omega=0$ and $\Re \mathfrak a[f]\geq0$, the form is called
accretive.
\end{definition}
 A nonnegative symmetric form is the
special case in which the sector has angle zero and is simply the positive
real axis.
From now on, assume for simplicity that $\mathfrak a$ is accretive. Define the form norm by
\[
  \|f\|_{\mathfrak a}^2
  =\|f\|^2+\Re \mathfrak a[f].
\]
The form is closed if $\mathcal D(\mathfrak a)$ is complete in this norm.  For
a form with a nonzero vertex one first makes a sufficiently large real shift.
The operator associated with a densely defined closed sectorial form is
defined weakly.  Namely, $f\in\mathcal D(A)$ precisely when
$f\in\mathcal D(\mathfrak a)$ and there exists $h\in\mathcal H$ such that
\begin{equation}
  \mathfrak a[f,g]=\langle h,g\rangle,
  \qquad g\in\mathcal D(\mathfrak a).
  \label{eq:associated-operator-definition}
\end{equation}
One then sets $Af=h$.  The vector $h$ is unique because
$\mathcal D(\mathfrak a)$ is dense.
The first representation theorem for sectorial forms states that the
operator $A$ constructed in this way is $m$-sectorial; see
\cite[Chapter~VI, Section~2]{Kato} or
\cite[Chapter~1]{OuhabazHeat}.  Here the letter $m$ means maximal.  In the
accretive case, maximality can be expressed by
\[
       \operatorname{Ran}(A+\lambda)=\mathcal H,
       \qquad \lambda>0.
\]
In particular, $-A$ generates a strongly continuous contraction semigroup by the Lumer-Phillips Theorem. We denote the corresponding semigroup by $(\mathrm e^{-tA})_{t\geq0}$.  Since $A$ is sectorial with angle smaller
than $\pi/2$, this semigroup is holomorphic in time in a complex sector \cite[Chapter IX, Theorem 1.24]{Kato}.
\subsection{The complex Schr\"odinger family}
\label{subsec:complex-schrodinger-family}

Let $Q:\mathbb R^d\to[0,\infty)$ be continuous and satisfy
$Q(x)\to\infty$ as $|x|\to\infty$.  On $\mathcal H=L^2(\mathbb R^d)$ recall that
\[
  \mathcal D
  =H^1(\mathbb R^d)\cap L^2(\mathbb R^d,Q(x)\,\de x)
\]
and, for $z\in\mathbb C$ with $\Re z>0$,
\begin{equation}
  \mathfrak h_z[f,g]
  =\frac12\int_{\mathbb R^d} \de x \,\nabla f(x)\cdot\overline{\nabla g(x)}
     +z\int_{\mathbb R^d}\de x \, Q(x)f(x)\overline{g(x)}.
  \nonumber
\end{equation}

\begin{proposition}
\label{prop:closed-sectorial-app}
For every $z$ with $\Re z>0$, the form $\mathfrak h_z$ is
densely defined, closed, and sectorial on the common domain $\mathcal D$.
Its associated operator, denoted by $H_z$, is the weak realization of
\[
       -\frac12\Delta+zQ.
\]
More precisely,
\[
 \mathcal D(H_z)
 =\left\{f\in\mathcal D:
 -\frac12\Delta f+zQf\in L^2(\mathbb R^d)
 \text{ in the sense of distributions}\right\}.
\]
Furthermore,  $H_z $ has compact resolvent and $H_z^*=H_{\overline z}$.
\end{proposition}

\begin{proof}
Since $Q$ is locally bounded,
$C_c^\infty(\mathbb R^d)\subset\mathcal D$, so the form is densely defined.
Write $z=\beta-iu$, where $\beta>0$.  Then
\begin{align}
    \nonumber
 \Re \mathfrak h_{\beta-iu}[f]
 &=\frac12\|\nabla f\|_2^2+\beta\int_{\mathbb R^d} \de x \, Q(x) |f(x)|^2,
 \\
 |\Im\mathfrak h_{\beta-iu}[f]|
 &=|u|\int_{\mathbb R^d} \de x \, Q(x)|f(x)|^2
 \leq\frac{|u|}{\beta}
       \Re \mathfrak h_{\beta-iu}[f].
 \nonumber
\end{align}
Thus the numerical range of the form lies in a sector of semi-angle at most
$\arctan(|u|/\beta)<\pi/2$.
The form norm is equivalent to
\[
  \|f\|_2^2+\|\nabla f\|_2^2+\int_{\mathbb R^d} \de x \, Q(x)|f(x)|^2.
\]
Completeness of $H^1(\mathbb R^d)$ and of the weighted space
$L^2(Q(x)\,\de x)$ therefore gives closedness. The description of the associated
operator follows from \eqref{eq:associated-operator-definition} by weak
integration by parts.
Since  $Q(x)\to\infty $ as  $|x|\to\infty $, Rellich compactness on balls together with control of the $L^2$-tail by the $Q$-term shows that the common form domain  $\mathcal D $ embeds compactly into  $L^2(\mathbb R^d) $. Consequently,  $H_z $ has compact resolvent for every  $\Re z>0 $.
Finally,
\[
  \overline{\mathfrak h_z[g,f]}=\mathfrak h_{\overline z}[f,g],
\]
which, together with the representation theorem, gives
$H_z^*=H_{\overline z}$.
\end{proof}

For real $z=\beta>0$, the form is symmetric and $H_\beta$ is self-adjoint.
For $u\neq0$, the operator $H_{\beta-iu}$ is generally not self-adjoint, but
it is still $m$-sectorial and therefore generates, as discussed in the previous subsection, the holomorphic semigroup
\[
      S_u(t)=\mathrm e^{-tH_{\beta-iu}}.
\]

\begin{lemma}
\label{lem:spectral-to-semigroup}
Let $K\subset\R$ be compact and let $a\in\R$.  Suppose that
\begin{equation}
 \inf_{u\in K}\inf_{\zeta\in\spec(H_{\beta-iu})}\Re\zeta>a
\label{eq:abstract-uniform-spectral-gap}
\end{equation}
and that
\begin{equation}
 \Re\mathfrak h_{\beta-iu}[f]
 \geq a\norm f_2^2,
 \qquad f\in\cD, u\in K.
 \label{eq:numerical-lower-bound}
\end{equation}
Then there are $C,\eta>0$ such that
\[
 \sup_{u\in K}\norm{S_u(t)}_{2\to2}
 \leq C \mathrm e^{-(a+\eta)t},
 \qquad t\geq1.
\]
The same conclusion holds for a norm-continuous family of complementary
spectral semigroups $S_u(t)(I-P_u)$, provided that the corresponding
complementary spectra satisfy \eqref{eq:abstract-uniform-spectral-gap} and
$\sup_{u\in K}\norm{P_u}_{2 \to 2}<\infty$.
\end{lemma}

\begin{proof}
The forms constitute a holomorphic family of type (a), and their associated
operators $H_{\beta-iu}$ form a holomorphic family of type (B); see
\cite[Chapter~VII, Section~4, especially Theorem~4.2]{Kato}.  We shall only
need norm continuity of the semigroup in the real parameter $u$, which also
follows directly from the Feynman-Kac formula. Indeed, for $f\in L^2(\mathbb R^d)$,
\[
 \Vert (S_u(t)-S_v(t))f\Vert_2
 \leq \frac{|u-v|}{\mathrm e\beta}\Vert P_t|f|\Vert_2,
 \qquad
 \norm{S_u(t)-S_v(t)}_{2\to2}
 \leq\frac{|u-v|}{\mathrm e\beta}.
\]
The semigroups are analytic, and analyticity together with compact resolvent
makes them immediately compact.  The spectral mapping theorem for 
norm-continuous, in particular analytic or immediately compact, semigroups
gives
\begin{equation}
 \spec(S_u(t))\setminus\{0\}
 =\{\mathrm e^{-t\zeta}:\zeta\in\spec(H_{\beta-iu})\};
 \label{eq:spectral-mapping}
\end{equation}
see \cite[Chapter~IV, Sections~2-3]{EngelNagelSemigroups}.

By compactness of $K$, the left-hand side of
\eqref{eq:abstract-uniform-spectral-gap} is at least $a+2\eta_0$ for some
$\eta_0>0$.  Hence \eqref{eq:spectral-mapping} implies
\[
 r(S_u(1))\leq \mathrm e^{-(a+2\eta_0)},
 \qquad u\in K,
\]
where $r$ denotes the spectral radius.
For each fixed $u$, the spectral-radius formula gives an integer $n_u\geq1$
such that
\[
 \norm{S_u(n_u)}_{2 \to 2}\leq \mathrm e^{-(a+\eta_0)n_u}.
\]
Norm continuity yields a neighborhood $V_u$ on which the right-hand side can
be replaced by $\mathrm e^{-(a+\eta_0/2)n_u}$.  Select a finite subcover
$V_{u_1},\ldots,V_{u_m}$ and put
$N=\max_j n_{u_j}$ and $n_*=\min_j n_{u_j}$.  The form bound
\eqref{eq:numerical-lower-bound} implies
\begin{equation}
 \norm{S_v(s)}_{2 \to 2}\leq \mathrm e^{-as},
 \qquad s\geq0, v\in K,
 \nonumber
\end{equation}
by the m-accretive representation theorem
\cite[Chapter~VI, Section~2]{Kato}.  If $v\in V_{u_j}$, the semigroup property
therefore gives
\[
 \norm{S_v(N)}_{2 \to 2}
 \leq \mathrm e^{-aN-(\eta_0/2)n_{u_j}}
 \leq \mathrm e^{-(a+\eta_1)N},
 \qquad
 \eta_1=\frac{\eta_0n_*}{2N}>0.
\]
Writing $t=kN+r$, with $0\leq r<N$, and iterating this estimate proves the
claim.  For the complementary family, put
$M_0=\sup_{u\in K}\norm{I-P_u}_{2 \to 2}<\infty$.  Its spectral-radius bound allows
the integers $n_u$ above to be chosen so large that the strict contraction at
time $n_u$ absorbs the factor $M_0$ occurring in the estimate
$\norm{S_u(s)(I-P_u)}_{2 \to 2}\leq M_0e^{-as}$.  The same finite-cover and block
iteration then proves the stated complementary estimate.
\end{proof}

\subsection{The complex Feynman-Kac formula}
\label{subsec:complex-fk-app}

Let $(W_t)_{t\geq0}$ be Brownian motion in $\mathbb R^d$ with generator
$\Delta/2$ and put
\[
       A_T=\int_0^T \de s \, Q(W_s).
\]
Because a Brownian path has compact range on every bounded time interval and
$Q$ is continuous, $A_T<\infty$ almost surely for every $T<\infty$.
For the real-valued version of the formula and its standard Schr\"odinger
semigroup consequences, see, for example, \cite{SimonSemigroups}.
The proof below records the extra points needed when the coupling is complex.

\begin{theorem}
\label{thm:complex-fk-app}
Let $z\in\mathbb C$ satisfy $\Re z>0$.  For every
$f\in L^2(\mathbb R^d)$ and $t>0$,
\begin{equation}
  (\mathrm e^{-tH_z}f)(x)
  =\mathbb E^{\prob_x}\left[
       \mathrm e^{-zA_t}f(W_t)
    \right]
  \label{eq:complex-fk-app}
\end{equation}
for almost every $x$.  The right-hand side defines a canonical pointwise
version whenever it is finite.  In particular, if $z=\beta-iu$, then
\begin{equation}
  |\mathrm e^{-tH_{\beta-iu}}f|(x)
  \leq (\mathrm e^{-tH_\beta}|f|)(x)
  \leq (P_t|f|)(x),
  \label{eq:complex-fk-domination-app}
\end{equation}
where $P_t=\mathrm e^{t\Delta/2}$ is the free Brownian semigroup.
\end{theorem}

\begin{proof}
We separate the proof into the bounded-potential case and the removal of the
boundedness assumption.  This makes clear that no positivity or
self-adjointness is being used for the imaginary part of the potential.
To start, 
let $Q_N=Q\wedge N$ and let $H_{z,N}$ be the operator associated with
\[
  \mathfrak h_{z,N}[f,g]
  =\frac12\int \de x \, \nabla f(x)\cdot\overline{\nabla g(x)}
   +z\int \de x \,  Q_N (x) f (x) \overline{g(x)},
  \qquad \mathcal D(\mathfrak h_{z,N})=H^1(\mathbb R^d).
\]
Since $Q_N$ is bounded, $zQ_N$ is a bounded multiplication operator and
\[
        H_{z,N}=-\frac12\Delta+zQ_N,
        \qquad \mathcal D(H_{z,N})=H^2(\mathbb R^d).
\]
Moreover, note that the Chernoff product formula (\cite[Theorem III.5.3]{EngelNagelSemigroups}) yields
\begin{equation}
 \mathrm e^{-tH_{z,N}}
 =\operatorname*{s-lim}_{m\to\infty}
 \left(P_{t/m}M_{\exp(-tzQ_N/m)}\right)^m
 \label{eq:trotter-bounded-complex-app}
\end{equation}
because
\[
 \frac{P_tM_{\mathrm e^{-tzQ_N}}f-f}{t}
 \longrightarrow \frac12\Delta f-zQ_Nf
 \quad\text{for }f\in C_c^\infty(\mathbb R^d).
\]
Here, we used $M_f$ to denote the multiplication operator
\[M_f :L^2(\bbR^d) \to L^2(\bbR^d), \quad g \mapsto fg.\]
Let $\delta=t/m$.  Iterating the Markov property shows that the operator in
the product on the right-hand side of
\eqref{eq:trotter-bounded-complex-app} equals
\begin{equation}
 \mathbb E^{\prob_x}\left[
  \exp\left(-z\delta\sum_{k=1}^m Q_N(W_{k\delta})\right)f(W_t)
 \right].
 \nonumber
\end{equation}
Since $s\mapsto Q_N(W_s)$ is almost surely continuous, its Riemann sums
converge to $\int_0^t \de s \, Q_N(W_s)$.  Moreover,
\[
 \left|
  \exp\left(-z\delta\sum_{k=1}^mQ_N(W_{k\delta})\right)
 \right|
 \leq1
\]
because $\Re z>0$ and $Q_N\geq0$.  Dominated convergence,
followed by the $L^2$-contractivity of $P_t$, therefore gives
\begin{equation}
 (\mathrm e^{-tH_{z,N}}f)(x)
 =\mathbb E^{\prob_x}\left[
   \mathrm e^{-zA_t^{(N)}}f(W_t)
  \right],
 \qquad
 A_t^{(N)}=\int_0^t \de s \, Q_N(W_s).
 \nonumber
\end{equation}
We now remove the cut-off on the external potential $Q$.
The forms $\mathfrak h_{z,N}$ are uniformly sectorial, because
\[
 |\Im\mathfrak h_{z,N}[f]|
 \leq\frac{|\Im z|}{\Re z}
       \Re \mathfrak h_{z,N}[f].
\]
Moreover, for $M\geq N$ the difference
\[
 \mathfrak h_{z,M}[f]-\mathfrak h_{z,N}[f]
 =z\int \de x \, (Q_M(x)-Q_N)(x)|f(x)|^2
\]
lies in the same fixed sector.  The monotone convergence theorem for closed
sectorial forms therefore applies; see
\cite{VogtVoigt}.  Its limiting form is
$\mathfrak h_z$, with domain
\[
 \left\{f\in H^1(\mathbb R^d):
       \sup_N\int \de x \, Q_N (x)|f(x)|^2<\infty\right\}
 =H^1(\mathbb R^d)\cap L^2(Q(x)\,\de x)=\mathcal D.
\]
Consequently, $\mathrm e^{-tH_{z,N}}f\longrightarrow \mathrm e^{-tH_z}f$ strongly in $L^2(\R^d)$
for every $t\geq0$ and $f\in L^2(\mathbb R^d)$.
Since $Q_N\uparrow Q$, it follows that $A_t^{(N)}\uparrow A_t<\infty$ almost surely.
Hence
\[
      \mathrm   e^{-zA_t^{(N)}}f(W_t)
       \longrightarrow \mathrm  e^{-zA_t}f(W_t)
       \quad\text{almost surely}.
\]
The absolute value is bounded by $|f(W_t)|$.  Thus, for almost every $x$,
dominated convergence gives
\[
 \mathbb E^{\prob_x}[\mathrm e^{-zA_t^{(N)}}f(W_t)]
 \longrightarrow
 \mathbb E^{\prob_x}[\mathrm e^{-zA_t}f(W_t)].
\]

Finally, $|\mathrm e^{-(\beta-iu)A_t}|= \mathrm  e^{-\beta A_t}$
gives the first inequality in
\eqref{eq:complex-fk-domination-app}; the second follows by discarding the
Feynman-Kac weight (recall $Q \ge 0$).
\end{proof}

The same argument gives two estimates used repeatedly in the paper.  First,
for $u,v\in\mathbb R$ and $f\in L^2(\mathbb R^d)$,
\begin{align}
    \nonumber
 |(S_u(t)-S_v(t))f|(x)
 &\leq |u-v|\,
 \mathbb E^{\prob_x}[A_t \mathrm e^{-\beta A_t}|f(W_t)|] 
 \leq \frac{|u-v|}{\mathrm e\beta}P_t|f|(x),
 \nonumber
\end{align}
because $a\mathrm e^{-\beta a}\leq(\mathrm e\beta)^{-1}$ for $a\geq0$.  Therefore $\|S_u(t)-S_v(t)\|_{2\to2}
  \leq\frac{|u-v|}{\mathrm e\beta}.$
Second, for fixed $t$ and $x$, the function
\[
      z\longmapsto Z_t(z,x)=\mathbb E^{\prob_x}[\mathrm e^{-zA_t}]
\]
is holomorphic on $\{\Re z>0\}$.  Indeed, on every compact
subset of that half-plane, differentiation under the expectation is justified
by $A_t^k \mathrm e^{-\delta A_t}\leq C_{k,\delta}$ for any $\delta >0$.

\subsection{Holomorphic form families and Riesz projections}
\label{subsec:holomorphic-riesz-app}

A family $(\mathfrak a_z)_{z\in U}$ of closed sectorial forms is called a
\textbf{holomorphic family of type~(a)} if the form domain is independent of $z$ and
\[
       z\longmapsto\mathfrak a_z[f,g]
\]
is holomorphic for every pair of form-domain vectors $f,g$.  The associated
operators \textbf{form a holomorphic family of type~(B)}. 
The defined family $z\mapsto\mathfrak h_z$ is of type~(a), since it has the common
domain $\mathcal D$ and depends affinely on $z$.  It follows that the resolvent
\[
       z\longmapsto(w-H_z)^{-1}
\]
is operator-norm holomorphic wherever it is defined; see
\cite[Chapter~VII, Section~4]{Kato}.
Let $\beta>0$ be real and let $\lambda=\lambda(\beta)$ be the simple ground
state eigenvalue of $H_\beta$.  Choose a positively oriented circle $\Gamma$ around $\lambda$ containing no other point of
$\operatorname{spec}(H_\beta)$.  For $z$ close to $\beta$, define
\begin{equation}
       P(z)=\frac1{2\pi i}\int_\Gamma \de w \, (w-H_z)^{-1},
       \nonumber
\end{equation}
the Riesz projection associated with the spectral points inside
$\Gamma$.  It is an operator-norm holomorphic projection.  Its rank is locally constant and equals one at $z=\beta$.  Hence, after decreasing the
neighborhood of $\beta$, there is exactly one eigenvalue $\lambda(z)$ inside $\Gamma$, counted with algebraic multiplicity; both $\lambda(z)$ and $P(z)$ are holomorphic.

\section{Spectral Analysis of Schrödinger operators and the ground state transform}
\begin{proposition}
\label{prop:ground-state-spectral-theory}
Let $Q\in\mathcal Q$ and $\beta>0$. Then the following statements hold.
\begin{enumerate}
    \item[(a)] $H_\beta$ is self-adjoint, bounded below, and has compact
    resolvent.
    \item[(b)] Its lowest eigenvalue $\lambda(\beta)$ is simple, and the
    associated normalized eigenfunction $\phi_\beta$ can be chosen strictly
    positive.
    \item[(c)] There exists $a=a(\beta,Q)>0$ such that
    \[
        \exp\bigl(a\langle x\rangle^{1+p/2}\bigr)\phi_\beta
        \in H^1(\mathbb R^d),
        \qquad \langle x\rangle=(1+|x|^2)^{1/2},
    \]
    where $p$ is as in assumption \eqref{bp:ass_2}. In particular,
    $\phi_\beta\in L^1(\mathbb R^d)\cap L^2(\mathbb R^d)$.
    \item[(d)] The map $\beta\mapsto\lambda(\beta)$ and the positive,
    $L^2(\mathbb R^d)$-normalized branch $\beta\mapsto\phi_\beta$ are real analytic on
    $(0,\infty)$.
\end{enumerate}
\end{proposition}

\begin{proof}
The first three assertions are classical; see e.g. \cite{LHB20}. To prove ~\textup{(d)}, note that for $\Re z>0$, the forms
\[
    h_z[f,g]
    =\frac12\int_{\mathbb R^d} \mathrm dx \, \nabla f(x)\cdot\overline{\nabla g(x)}
      +z\int_{\mathbb R^d}\mathrm dx \, Q(x)f(x)\overline{g(x)}
\]
form a holomorphic family with the common domain $\mathcal D$. Their associated
operators constitute a holomorphic family of type~\textup{(B)}; see
\cite[Chapter~VII, Section~4]{Kato}. Since $\lambda(\beta)$ is isolated and
simple, analytic perturbation theory gives, locally about every
$\beta_0>0$, an analytic eigenvalue, an analytic rank-one eigenprojection,
and an analytic normalized eigenvector. On the real axis the normalized
ground state is uniquely determined by the requirement that it be positive.
The local branches therefore agree on overlaps and give the claimed global
real-analytic branches on $(0,\infty)$.
\end{proof}

For the following result, recall that $I_\beta = \lambda(\beta) - \beta \theta_\beta$.

\begin{proposition}
\label{prop:hellmann-feynman}
For every $\beta > 0$, the ground state energy $\lambda(\beta)$ is differentiable with respect to $\beta$, and
\begin{equation}
\lambda'(\beta) = \int_{\mathbb{R}^d} \de x \, Q(x) \phi_{\beta}(x)^2  =: \theta_{\beta}.
\label{eq:hellmann_feynman}
\end{equation}
Let 
\[\mathcal{D}_{\beta,0} = \{g \in \mathcal{D} : g \text{ is real-valued and } \langle g, \phi_{\beta}\rangle = 0\}\]
and define on this real Hilbert space the coercive form $b_{\beta}[f,g] = h_{\beta}[f,g] - \lambda(\beta)\langle f,g\rangle$ and the bounded linear functional $L_{\beta}(g) = \int_{\mathbb{R}^d} \de x \, (Q(x) - \theta_{\beta})\phi_{\beta}(x)g(x) $.
Then, there exists a unique $u_{\beta} \in \mathcal{D}_{\beta,0}$ satisfying
\begin{equation}
b_{\beta}[u_{\beta}, g] = L_{\beta}(g) \quad \text{for all } g \in \mathcal{D}_{\beta,0}.
\label{eq:lax-milgram-u}
\end{equation}
Furthermore, $\beta \mapsto \phi_{\beta}$ is differentiable in the form norm $\|\cdot\|_{\mathcal{D}}$ with $\partial_{\beta}\phi_{\beta} = -u_{\beta}$, and $\lambda(\beta)$ is twice differentiable with
\begin{equation}
\lambda''(\beta) = -2 b_{\beta}[u_{\beta}, u_{\beta}] < 0.
\nonumber
\end{equation}
In particular, $\beta \mapsto \theta_{\beta}$ is strictly decreasing, $\sigma_{\beta}^2 = -\lambda''(\beta) > 0$, and $I_{\beta} = \frac{1}{2} \int_{\mathbb{R}^d} \de x \, |\nabla\phi_{\beta}(x)|^2$.
\end{proposition}

\begin{proof}
We begin by proving \eqref{eq:hellmann_feynman}.
Fix $\beta > 0$ and let $\delta > 0$ be a real perturbation. Let $\phi_{\beta}$ and $\phi_{\beta+\delta}$ be the unique positive, $L^2$-normalized ground states for $H_{\beta}$ and $H_{\beta+\delta}$, with eigenvalues $\lambda(\beta)$ and $\lambda(\beta+\delta)$, respectively. By the Rayleigh-Ritz variational principle for $h_{\beta+\delta}$ it holds that
\[
\lambda(\beta+\delta) = h_{\beta+\delta}[\phi_{\beta+\delta}] = h_{\beta}[\phi_{\beta+\delta}] + \delta \int_{\mathbb{R}^d} \de x \, Q(x) \phi_{\beta+\delta}^2 (x).
\]
Testing $h_{\beta+\delta}$ with the trial function $\phi_{\beta} \in \mathcal{D}$ yields
\[
\lambda(\beta+\delta) \le h_{\beta+\delta}[\phi_{\beta}] = \lambda(\beta) + \delta \int_{\mathbb{R}^d} \de x \, Q(x) \phi_{\beta}^2(x).
\]
Rearranging gives $\frac{\lambda(\beta+\delta) - \lambda(\beta)}{\delta} \le \int_{\mathbb{R}^d} \de x \, Q(x) \phi_{\beta}^2 (x)$ for $\delta > 0$. Reversing the roles of $\beta$ and $\beta+\delta$ yields the lower bound $\int_{\mathbb{R}^d} \de x \, Q(x) \phi_{\beta+\delta}^2 (x) \le \frac{\lambda(\beta+\delta) - \lambda(\beta)}{\delta}$. Combining yields
\begin{equation}
\int_{\mathbb{R}^d} \de x \, Q(x) \phi_{\beta+\delta}^2 (x)  \le \frac{\lambda(\beta+\delta) - \lambda(\beta)}{\delta} \le \int_{\mathbb{R}^d} \de x \, Q(x) \phi_{\beta}^2 (x).
\label{eq:rr-sandwich}
\end{equation}
Since $\lambda(\beta+\delta)$ is bounded near $\delta = 0$, the family $\{\phi_{\beta+\delta}\}$ is uniformly bounded in the form domain $\mathcal{D}$. Compactness of the embedding $\mathcal{D}\hookrightarrow L^2(\mathbb{R}^d)$ and simplicity of the positive ground state imply that $\phi_{\beta+\delta}\to\phi_{\beta}$ strongly in $L^2(\mathbb{R}^d)$ as $\delta\to0$. Fatou's Lemma on $Q\phi_{\beta+\delta}^2$ together with \eqref{eq:rr-sandwich} yields $\lim_{\delta\to0}\int_{\mathbb{R}^d}\de x\,Q(x)\phi_{\beta+\delta}^2(x)=\int_{\mathbb{R}^d}\de x\,Q(x)\phi_{\beta}^2(x)$. Taking $\delta\to0$ in~\eqref{eq:rr-sandwich} establishes differentiability of $\lambda(\beta)$ and proves~\eqref{eq:hellmann_feynman}.

For future use we improve the strong convergence of $\phi_{\beta+\delta} \to \phi_\beta$ in $L^2(\bbR^d)$ to strong convergence in the weighted $L^2$-space $L^2(\R^d,Q(x)\,\de x)$. We have seen already that $\phi_{\beta+\delta}\rightharpoonup\phi_\beta$ in $L^2(\mathbb R^d,Q(x)\,\de x)$. 
Together with
\[
 \int_{\mathbb R^d}\de x \, Q(x)\phi_{\beta+\delta}(x)^2
 \longrightarrow
 \int_{\mathbb R^d} \de x \,Q(x)\phi_\beta(x)^2,
\]
the Hilbert-space characterization of strong convergence therefore yields $\phi_{\beta+\delta}\longrightarrow\phi_\beta$ strongly in $L^2(\mathbb R^d,Q(x)\,\de x)$.

By the spectral gap of $H_\beta$, $\lambda_1(\beta) - \lambda(\beta) > 0$, which yields that the quadratic form $b_{\beta}[g,g] \ge (\lambda_1(\beta) - \lambda(\beta))\|g\|_2^2$ is coercive and induces an equivalent norm on the closed real Hilbert subspace $\mathcal{D}_{\beta,0}$. By Cauchy-Schwarz, $L_{\beta}$ is a bounded linear functional on $\mathcal{D}_{\beta,0}$. The Lax-Milgram theorem guarantees a unique $u_{\beta} \in \mathcal{D}_{\beta,0}$ satisfying~\eqref{eq:lax-milgram-u}.

Next, set $w_{\delta} = \frac{\phi_{\beta+\delta} - \phi_{\beta}}{\delta}$ and decompose it as $w_{\delta} = w_{\delta}^\perp + c_{\delta} \phi_{\beta}$ with $w_{\delta}^\perp \in \mathcal{D}_{\beta,0}$. Subtracting the weak eigenvalue equations for $H_{\beta+\delta}$ and $H_{\beta}$ tested against $g \in \mathcal{D}_{\beta,0}$ yields
\[
b_{\beta}[w_{\delta}^\perp, g] = -\int_{\mathbb{R}^d} \de x \, \left(Q(x) - \frac{\lambda(\beta+\delta) - \lambda(\beta)}{\delta}\right) \phi_{\beta+\delta} (x) \, g(x).
\]
As $\delta \to 0$, the right-hand side converges uniformly on unit form-norm vectors $g$ to $-L_{\beta}(g)$. Coercivity of $b_{\beta}$ implies $w_{\delta}^\perp \to -u_{\beta}$ strongly in $\mathcal{D}$. 
It remains to
control the component parallel to $\phi_\beta$. Since
\[
 \delta c_\delta
 =\langle\phi_{\beta+\delta}-\phi_\beta,\phi_\beta\rangle
 \longrightarrow 0,
\]
and $\|\phi_{\beta+\delta}\|_2=\|\phi_\beta\|_2=1$, we have
\[
 0
 =2c_\delta+\delta\|w_\delta\|_2^2
 =2c_\delta+\delta\bigl(\|w_\delta^\perp\|_2^2+c_\delta^2\bigr).
\]
Thus
\[
 c_\delta(2+\delta c_\delta)
 =-\delta\|w_\delta^\perp\|_2^2.
\]
Since $w_\delta^\perp$ is bounded in $\mathcal D$ and
$\delta c_\delta\to0$, it follows that $c_\delta\to0$. Hence
$w_\delta\to-u_\beta$ strongly in $\mathcal D$, proving
$\partial_\beta\phi_\beta=-u_\beta$.

To compute $\lambda''(\beta)$, we again consider the difference quotient of $\lambda'(\beta) = \theta_{\beta}$ for which we already found that
\[
\frac{\lambda'(\beta+\delta) - \lambda'(\beta)}{\delta} = \int_{\mathbb{R}^d} \de x \, Q(x) \, w_{\delta}(x) \, \big(\phi_{\beta+\delta}(x) + \phi_{\beta}(x) \big).
\]
Since $w_{\delta} \to -u_{\beta}$ and $\phi_{\beta+\delta} \to \phi_{\beta}$ strongly in $L^2(\mathbb{R}^d,Q(x)\,\de x)$, taking $\delta \to 0$ yields
\[
\lambda''(\beta) = -2 \int_{\mathbb{R}^d} \de x \, Q(x) \, u_{\beta}(x) \, \phi_{\beta}(x) .
\]
Since $u_{\beta} \in \mathcal{D}_{\beta,0}$, we have $\theta_{\beta} \int_{\mathbb{R}^d} \de x \, u_{\beta}(x)\phi_{\beta}(x) = 0$. Subtracting this zero term gives
\[
\lambda''(\beta) = -2 \int_{\mathbb{R}^d} \de x \, \big(Q(x) - \theta_{\beta}\big)\phi_{\beta}(x) u_{\beta}(x)  = -2 L_{\beta}(u_{\beta}) = -2 b_{\beta}[u_{\beta}, u_{\beta}],
\]
where the last equality follows from taking $g = u_{\beta}$ in~\eqref{eq:lax-milgram-u}.
By coercivity, $b_{\beta}[u_{\beta}, u_{\beta}] \ge 0$. If $u_{\beta} = 0$, then $L_{\beta}$ vanishes on $\mathcal{D}_{\beta,0}$ and on $\phi_{\beta}$, so $L_{\beta} = 0$ on all of $\mathcal{D}$. Testing against $C_c^\infty(\mathbb{R}^d)$ would give $(Q - \theta_{\beta})\phi_{\beta} = 0$ a.e., forcing $Q$ to be constant since $\phi_{\beta} > 0$ everywhere. This contradicts the assumptions on $Q$. Thus $u_{\beta} \neq 0$, proving $\lambda''(\beta) < 0$.

Finally, taking the inner product of the eigenvalue equation $H_\beta \phi_\beta = \lambda(\beta)\phi_\beta$ with $\phi_\beta$ gives $\lambda(\beta) = \frac{1}{2}\|\nabla\phi_\beta\|_2^2 + \beta \theta_\beta$, which rearranges directly to $I_\beta = \lambda(\beta) - \beta \theta_\beta = \frac{1}{2}\|\nabla\phi_\beta\|_2^2$.
\end{proof}

Recall that
\[
    q_*=\min_{x\in\mathbb R^d}Q(x),
    \qquad
    \theta_\beta=\lambda'(\beta)
      =\int_{\mathbb R^d} \mathrm dx \, Q(x)\phi_\beta(x)^2.
\]
The following result justifies the definition of $\beta_\theta$ used in the
main theorem.

\begin{lemma}
\label{lem:range-constraint-coupling}
Assume that $Q\colon\mathbb R^d\to[0,\infty)$ is continuous,
$Q(x)\to\infty$ as $|x|\to\infty$, and is not constant.  Then
\[
    \lim_{\beta\downarrow0}\lambda'(\beta)=\infty,
    \qquad
    \lim_{\beta\uparrow\infty}\lambda'(\beta)=q_*.
\]
Consequently,
\[
    \beta\longmapsto\lambda'(\beta)
\]
is a continuous strictly decreasing bijection from $(0,\infty)$ onto
$(q_*,\infty)$.  In particular, for every $\theta>q_*$ there is a unique
$\beta_\theta>0$ such that $\lambda'(\beta_\theta)=\theta$.
\end{lemma}

\begin{proof}
Continuity and strict monotonicity of $\lambda'$ follow from the real
analyticity of $\lambda$ and the strict inequality $\lambda''(\beta)<0$
proved above.  Moreover,
\[
    \lambda'(\beta)-q_*
    =\int_{\mathbb R^d}\de x \, (Q(x)-q_*)\phi_\beta(x)^2>0.
\]
Indeed, $Q-q_*\geq0$, the ground state is strictly positive, and $Q$ is not
constant.  It remains to identify the two endpoint limits.
We first prove that
\begin{equation}
    \lambda(\beta)\longrightarrow0
    \qquad\text{as }\beta\downarrow0.
\label{eq:lambda-small-beta}
\end{equation}
Choose $\eta\in C_c^\infty(B(0,1))$ with $\|\eta\|_2=1$ and put
$f_L(x)=L^{-d/2}\eta(x/L)$.
Then $\|f_L\|_2=1$ and
\[
    \frac12\|\nabla f_L\|_2^2
    =\frac{1}{2L^2}\|\nabla\eta\|_2^2,
    \qquad
    \int_{\mathbb R^d} \de x \, Q(x)|f_L(x)|^2
    \leq \sup_{|x|\leq L}Q(x)<\infty.
\]
By Rayleigh-Ritz,
\[
    0\leq\lambda(\beta)
    \leq \frac{1}{2L^2}\|\nabla\eta\|_2^2
       +\beta\sup_{|x|\leq L}Q(x).
\]
Given $\varepsilon>0$, we first choose $L$ so that the kinetic term is at
most $\varepsilon/2$, and then choose $\beta$ so that the potential term is
at most $\varepsilon/2$.  This proves \eqref{eq:lambda-small-beta}.

Suppose now, contrary to the first claimed limit, that there are
$\beta_n\downarrow0$ and $M<\infty$ such that $\theta_{\beta_n}=\lambda'(\beta_n)\leq M$ for every $n$.
Write $\phi_n=\phi_{\beta_n}$.  Taking the inner product of the
ground-state equation with $\phi_n$ gives
\[
    \lambda(\beta_n)
    =\frac12\|\nabla\phi_n\|_2^2
      +\beta_n\theta_{\beta_n}.
\]
In view of \eqref{eq:lambda-small-beta} it must hold that
\begin{equation}
    \|\nabla\phi_n\|_2\longrightarrow0.
\label{eq:derivative-vanishes}
\end{equation}
On the other hand, the bound
\[
    \int_{\mathbb R^d}\de x \, Q(x) \phi_n(x)^2 \leq M
\]
makes the probability measures $\phi_n(x)^2\,\mathrm dx$ tight, because
$Q(x)\to\infty$.  In particular, one can choose $R<\infty$, independently
of $n$, such that
\begin{equation}
    \int_{B(0,R)} \de x \, \phi_n(x)^2\geq\frac12.
\label{eq:tight-mass}
\end{equation}
For each $m\geq R$, the sequence $(\phi_n)$ is bounded in $H^1(B(0,m))$.
By Rellich compactness and a diagonal subsequence, there is
$\phi\in L^2_{\mathrm{loc}}(\mathbb R^d)$ such that
$\phi_n\to\phi$ strongly in $L^2(B(0,m))$ for every $m$.  Equation
\eqref{eq:derivative-vanishes} implies $\nabla\phi=0$ in the distributional
sense.  Since $\mathbb R^d$ is connected, $\phi=c$ almost everywhere for
some constant $c$.  The mass bound \eqref{eq:tight-mass} gives
$|c|^2|B(0,R)|\geq1/2$, so $c\neq0$.  On the other hand, local $L^2$
convergence and $\|\phi_n\|_2=1$ give
$|c|^2|B(0,m)|\leq1$ for every $m$, a contradiction as $m\to\infty$.
Hence
$\lambda'(\beta)\to\infty$ as $\beta\downarrow0$.
We finally consider $\beta\to\infty$.  Fix $\varepsilon>0$ and choose
$x_*\in\mathbb R^d$ with $Q(x_*)=q_*$.  By continuity, there is a nonempty
open neighborhood $U_\varepsilon$ containing $x_*$ such that $Q(x)\leq q_*+\varepsilon$ for $x\in U_\varepsilon$.
Choose $f_\varepsilon\in C_c^\infty(U_\varepsilon)$ with
$\|f_\varepsilon\|_2=1$ and set $K_\varepsilon=\frac12\|\nabla f_\varepsilon\|_2^2$.
Rayleigh-Ritz gives
\begin{equation}
    \lambda(\beta)
    \leq K_\varepsilon+\beta(q_*+\varepsilon).
\label{eq:large-beta-upper}
\end{equation}
Since $Q\geq q_*$, we also have
\begin{equation}
    \lambda(\beta/2)\geq\frac\beta2q_*.
\label{eq:large-beta-lower}
\end{equation}
The function $\lambda$ is concave, being the infimum of affine functions of
$\beta$.  Therefore its derivative at the right endpoint is bounded above by
\[
    \lambda'(\beta)
    \leq \frac{\lambda(\beta)-\lambda(\beta/2)}{\beta/2}.
\]
Using \eqref{eq:large-beta-upper} and \eqref{eq:large-beta-lower}, we obtain
\[
    q_*<\lambda'(\beta)
    \leq q_*+2\varepsilon+\frac{2K_\varepsilon}{\beta}.
\]
Taking first $\beta\to\infty$ and then $\varepsilon\downarrow0$ proves $\lim_{\beta\to\infty}\lambda'(\beta)=q_*$.
The endpoint limits, continuity, and strict monotonicity complete the proof.
\end{proof}

\begin{proposition}
\label{prop:ground-state-transform-rigorous}
The process
\begin{equation}
 M_t^\beta
 = \mathrm e^{\lambda(\beta)t-\beta A_t}
   \frac{\phi_\beta(W_t)}{\phi_\beta(x)},
 \qquad t\ge0,
 \nonumber
\end{equation}
is a positive mean-one  $(\mathcal F_t) $-martingale under
 $\mathbb P_x $.  The consistent measures
\[
 \de \mathbb P_x^\beta|_{\mathcal F_t}
 =M_t^\beta\,\de \mathbb P_x|_{\mathcal F_t}
\]
define the law of the diffusion
\begin{equation}
 \de X_t=\de B_t+\nabla\log\phi_\beta(X_t)\,\de t,
 \qquad X_0=x.
 \label{eq:ground-state-sde-rigorous}
\end{equation}
\end{proposition}

\begin{proof}
    See \cite[Chapters 4 and 5]{LHB20}.
\end{proof}

\section{A Stone Lemma}

We use the Fourier-transform convention
\[
 \widehat h(u)=\int_\R \de y \, \mathrm e^{iuy}h(y).
\]
For completeness, we state the precise smoothing lemma needed below.  If
$h\geq0$, set
\[
 U_d(h)=d\sum_{k\in\mathbb Z}\sup_{[kd,(k+1)d)}h,
 \qquad
 L_d(h)=d\sum_{k\in\mathbb Z}\inf_{[kd,(k+1)d)}h.
\]
We call $h$ directly Riemann-integrable if these sums are finite for some
mesh and both converge to $\int_\R \de x \, h(x)$ as $d\downarrow0$.  A real-valued $h$
is directly Riemann-integrable if both $h^+$ and $h^-$ have this property.
\begin{lemma}
\label{lem:stone-smoothing}
Let $(\mu_{T,x})_{T\geq1,x\in\mathbb R^d}$ be probability measures on $\mathbb R$, with
characteristic functions $\chi_{T,x}$.  Suppose that, for every compact
$K\subset\mathbb R^d$, the following hold for some $\sigma^2>0$.
\begin{enumerate}
 \item Locally uniformly in $(x,v)$, \label{equ:stone_1}
       \[
       \chi_{T,x}(v/\sqrt T)\longrightarrow \mathrm e^{-\sigma^2v^2/2}.
       \]
 \item There are $\varepsilon,c,C,\gamma>0$ such that for $T$ large enough
       \[
       \sup_{x\in K}|\chi_{T,x}(u)|
       \leq C\bigl(\mathrm e^{-cTu^2}+ \mathrm e^{-\gamma T}\bigr),
       \qquad |u|\leq\varepsilon.
       \]\label{equ:stone_2}
 \item For every $0<\delta<R<\infty$, there are $C_{K,\delta,R},\eta>0$
       such that for $T$ large enough
       \[
       \sup_{\substack{x\in K\\\delta\leq|u|\leq R}}
       |\chi_{T,x}(u)|
       \leq C_{K,\delta,R}\mathrm e^{-\eta T}.
       \]\label{equ:stone_3}
\end{enumerate}
Then, for every directly Riemann-integrable $h$ and every $M<\infty$,
\begin{equation}
 \sup_{\substack{x\in K\\|a|\leq M}}
 \left|
  \sqrt T\int_\R \mu_{T,x}(\de y) \, h(y-a)
  -\frac1{\sqrt{2\pi\sigma^2}}
   \int_\R \de y \, h(y)
 \right|
 \longrightarrow0.
 \label{eq:abstract-stone-conclusion}
\end{equation}
\end{lemma}

\begin{proof}
We first treat $h\in L^1(\R)$ such that
$\widehat h\in C_c(\R)$.  If
$\operatorname{supp}\widehat h\subset[-R,R]$, Fourier inversion gives
\begin{align*}
 \sqrt T\int_\R \mu_{T,x}(\de y) \, h(y-a)
 =
 \frac{\sqrt T}{2\pi}
 \int_{-R}^R \de u \, \mathrm e^{iua}\widehat h(u)\chi_{T,x}(-u).
\end{align*}
On $|u|\leq\varepsilon$, substitute $v=\sqrt T\,u$.  Assumptions \eqref{equ:stone_1}-\eqref{equ:stone_2}
and dominated convergence show, uniformly for $x\in K$ and $|a|\leq M$,
that this part converges to
\[
 \frac{\widehat h(0)}{2\pi}
 \int_\R \de v \, \mathrm e^{-\sigma^2v^2/2}
 =\frac1{\sqrt{2\pi\sigma^2}}\int_\R \de x \, h(x).
\]
The contribution of $\varepsilon\leq|u|\leq R$ tends to zero by Assumption \eqref{equ:stone_3}.
Thus \eqref{eq:abstract-stone-conclusion} holds for band-limited test
functions.

We record two standard consequences.  First, using Beurling-Selberg majorants, for every bounded interval $I$ and every
$\rho>0$, there exist integrable functions $b^-_{I,\rho}$ and
$b^+_{I,\rho}$ with compactly supported Fourier transforms such that
\[
 b^-_{I,\rho}\leq\mathbf 1_I\leq b^+_{I,\rho},
 \qquad
 \int_\R \de x \, \big(b^+_{I,\rho} (x) -b^-_{I,\rho} (x)\big)\leq\rho.
\]
This is quite classical; see e.g. the works of Stone
\cite{StoneLLT,StoneRatio}.  Applying the already proved band-limited limit
and then letting $\rho\downarrow0$ gives, for every bounded interval $I$,
\begin{equation}
 \sqrt T\,\mu_{T,x}(I+a)
 \longrightarrow
 \frac{|I|}{\sqrt{2\pi\sigma^2}},
 \label{eq:interval-llt}
\end{equation}
locally uniformly in $(x,a)$.

Second, we need a uniform concentration bound.  Choose a nonnegative
$\kappa\in L^1(\R)$ such that $\int_\R \de x \, \kappa (x) =1$ and
$\widehat\kappa$ has compact support; for example, a normalized version of
$(\sin y/y)^2$.  Fix $r>0$ and put
$m_r=\int_{-r}^r \de y \,\kappa(y)>0$.  If $I=[b,b+d]$, then
\[
 G_I(y)=m_r^{-1}
 \bigl(\mathbf 1_{[b-r,b+d+r]}*\kappa\bigr)(y)
\]
is nonnegative, has compactly supported Fourier transform, and satisfies
$G_I\geq\mathbf 1_I$.  Applying the Fourier estimate used in the first paragraph,
but taking absolute values rather than a limit, yields
\begin{equation}
 \sup_{T\geq1}\sup_{x\in K}\sup_{b\in\bbR}
 \sqrt T\,\mu_{T,x}([b,b+d])<\infty
 \label{eq:uniform-concentration}
\end{equation}
for every fixed $d>0$.  Crucially, the translation $b$ contributes only a
complex factor of modulus one, so the bound is uniform over all $b$.

Let now $h\geq0$ be directly Riemann-integrable.  For a mesh $d>0$, put
\[
 I_k=[kd,(k+1)d),\qquad
 m_k=\inf_{I_k}h,\qquad M_k=\sup_{I_k}h.
\]
Then
\[
 \sum_km_k\mu_{T,x}(I_k+a)
 \leq\int \mu_{T,x}(\de y) \, h(y-a)
 \leq\sum_kM_k\mu_{T,x}(I_k+a).
\]
For a finite number of blocks, \eqref{eq:interval-llt} applies uniformly for
$x\in K$ and $|a|\leq M$.  The tails are uniformly negligible by
\eqref{eq:uniform-concentration} and the summability of
$(M_k)_{k\in\mathbb Z}$.  Consequently,
\begin{align*}
 \frac{d}{\sqrt{2\pi\sigma^2}}\sum_km_k
 &\leq \liminf_{T\to\infty}
       \sqrt T\int \mu_{T,x}(\de y) \, h(y-a)\\
 &\leq \limsup_{T\to\infty}
       \sqrt T\int \mu_{T,x}(\de y) \, h(y-a)
 \leq\frac{d}{\sqrt{2\pi\sigma^2}}\sum_kM_k,
\end{align*}
with the inequalities uniform on $x\in K$, $|a|\leq M$.  Letting
$d\downarrow0$ proves the assertion for nonnegative $h$.  Applying the result
to the positive and negative parts proves it for every directly
Riemann-integrable $h$.
\end{proof}

\end{appendix}

\end{document}

%% file: commands.tex
\newcommand{\bea}{\begin{eqnarray}}
\newcommand{\eea}{\end{eqnarray}}
\newcommand{\<}{\langle}
\renewcommand{\>}{\rangle}

\newcommand\eg{{\text{\eg~}}}

\def\spec{{\rm spec}}

\def\cF{{\mathcal F}}

\def\Z{{\mathbb Z}}

\def\R{{\mathbb R}}

\def\Z{{\mathbb Z}}

\def\de{{\rm d}}

\def\prob{{\mathbb P}}

\def\<{\langle}
\def\>{\rangle}

\def\cL{{\cal L}}

\def\cD{{\cal D}}

\def\b0{{\boldsymbol{0}}}

\def\spec{{\mathrm {spec}}}

\def\de{{\mathrm {d}}}

\def\cD{{{\mathcal D}}}

\renewcommand{\b}{\mathbf{b}}

\def\lt{\left}
\def\rt{\right}

\def\bbC{{\mathbb{C}}}
\def\bbE{{\mathbb{E}}}

\def\bbP{{\mathbb{P}}}
\def\bbPh{{\hat{\mathbb{P}}}}
\def\bbR{{\mathbb{R}}}

\def\cF{{\mathcal{F}}}

\def\cQ{{\mathcal{Q}}}

\def\cL{{\mathcal{L}}}

\newcommand{\norm}[1]{{\lt\|#1\rt\|}}

\newcommand{\ip}[2]{\left\langle #1,#2\right\rangle}